\documentclass{article}

\usepackage[british]{babel}
\usepackage[latin1]{inputenc}
\usepackage[T1]{fontenc}
\usepackage{lmodern}
\usepackage{amsmath}
\usepackage{amssymb}
\usepackage{amsthm}
\usepackage{mathtools}
\usepackage{upgreek}
\usepackage{isomath}
\usepackage{paralist}
\usepackage{hyperref}
\usepackage{ellipsis}
\usepackage{fixltx2e}
\usepackage{mparhack}

\makeatletter
\def\star@command@name#1{#1@star}
\def\nostar@command@name#1{#1@nostar}
\def\newstarcommand{%
  \newif\if@star@newstarcommand@
  \@ifstar{\@star@newstarcommand@true\@newstarcommand}{\@star@newstarcommand@false\@newstarcommand}%
}
\def\@newstarcommand#1{%
  \count@=\escapechar
  \escapechar=\m@ne
  \edef\@newstarcommand@command@string{\string#1}%
  \escapechar=\count@
  \def\@newstarcommand@command{#1}%
  \expandafter\def\expandafter\@newstarcommand@command@star\expandafter{\csname\star@command@name\@newstarcommand@command@string\endcsname}%
  \expandafter\def\expandafter\@newstarcommand@command@nostar\expandafter{\csname\nostar@command@name\@newstarcommand@command@string\endcsname}%
  \expandafter\expandafter\expandafter\expandafter\expandafter\expandafter\expandafter
  \newcommand\expandafter\expandafter\expandafter\expandafter\expandafter\expandafter\expandafter
  *\expandafter\expandafter\expandafter\expandafter\expandafter\expandafter\expandafter{%
    \expandafter\expandafter\expandafter\@newstarcommand@command\expandafter\expandafter\expandafter
  }\expandafter\expandafter\expandafter{%
    \expandafter\expandafter\expandafter\@ifstar\expandafter\@newstarcommand@command@star\@newstarcommand@command@nostar
  }%
  \@testopt\@@newstarcommand0%
}
\def\@@newstarcommand[#1]{%
  \@ifnextchar[{\@@newstarcommand@opt[{#1}]}{\@@newstarcommand@noopt[{#1}]}%
}
\def\@@newstarcommand@noopt[#1]#2#3{%
  \if@star@newstarcommand@
    \expandafter\newcommand\expandafter*\@newstarcommand@command@star[{#1}]{#2}%
    \expandafter\newcommand\expandafter*\@newstarcommand@command@nostar[{#1}]{#3}%
  \else
    \expandafter\newcommand\@newstarcommand@command@star[{#1}]{#2}%
    \expandafter\newcommand\@newstarcommand@command@nostar[{#1}]{#3}%
  \fi
}
\def\@@newstarcommand@opt[#1][#2]#3#4{%
  \if@star@newstarcommand@
    \expandafter\newcommand\expandafter*\@newstarcommand@command@star[{#1}][{#2}]{#3}%
    \expandafter\newcommand\expandafter*\@newstarcommand@command@nostar[{#1}][{#2}]{#4}%
  \else
    \expandafter\newcommand\@newstarcommand@command@star[{#1}][{#2}]{#3}%
    \expandafter\newcommand\@newstarcommand@command@nostar[{#1}][{#2}]{#4}%
  \fi
}
\makeatother

\makeatletter
\newcommand*\if@single[3]{%
  \setbox0\hbox{${\mathaccent"0362{#1}}^H$}%
  \setbox2\hbox{${\mathaccent"0362{\kern0pt#1}}^H$}%
  \ifdim\ht0=\ht2 #3\else #2\fi
  }
\newcommand*\rel@kern[1]{\kern#1\dimexpr\macc@kerna}
\newcommand*\widebar[1]{\@ifnextchar^{{\wide@bar{#1}{0}}}{\wide@bar{#1}{1}}}
\newcommand*\wide@bar[2]{\if@single{#1}{\wide@bar@{#1}{#2}{1}}{\wide@bar@{#1}{#2}{2}}}
\newcommand*\wide@bar@[3]{%
  \begingroup
  \def\mathaccent##1##2{%
    \if#32 \let\macc@nucleus\first@char \fi
    \setbox\z@\hbox{$\macc@style{\macc@nucleus}_{}$}%
    \setbox\tw@\hbox{$\macc@style{\macc@nucleus}{}_{}$}%
    \dimen@\wd\tw@
    \advance\dimen@-\wd\z@
    \divide\dimen@ 3
    \@tempdima\wd\tw@
    \advance\@tempdima-\scriptspace
    \divide\@tempdima 10
    \advance\dimen@-\@tempdima
    \ifdim\dimen@>\z@ \dimen@0pt\fi
    \rel@kern{0.6}\kern-\dimen@
    \if#31
      \overline{\rel@kern{-0.6}\kern\dimen@\macc@nucleus\rel@kern{0.4}\kern\dimen@}%
      \advance\dimen@0.4\dimexpr\macc@kerna
      \let\final@kern#2%
      \ifdim\dimen@<\z@ \let\final@kern1\fi
      \if\final@kern1 \kern-\dimen@\fi
    \else
      \overline{\rel@kern{-0.6}\kern\dimen@#1}%
    \fi
  }%
  \macc@depth\@ne
  \let\math@bgroup\@empty \let\math@egroup\macc@set@skewchar
  \mathsurround\z@ \frozen@everymath{\mathgroup\macc@group\relax}%
  \macc@set@skewchar\relax
  \let\mathaccentV\macc@nested@a
  \if#31
    \macc@nested@a\relax111{#1}%
  \else
    \def\gobble@till@marker##1\endmarker{}%
    \futurelet\first@char\gobble@till@marker#1\endmarker
    \ifcat\noexpand\first@char A\else
      \def\first@char{}%
    \fi
    \macc@nested@a\relax111{\first@char}%
  \fi
  \endgroup
}
\makeatother

\theoremstyle{plain}
\newtheorem{theorem}{Theorem}
\newtheorem{corollary}[theorem]{Corollary}
\newtheorem{proposition}[theorem]{Proposition}
\newtheorem{lemma}[theorem]{Lemma}
\theoremstyle{definition}
\newtheorem{definition}[theorem]{Definition}

\theoremstyle{remark}
\newtheorem{remark}[theorem]{Remark}

\theoremstyle{plain}

\newcommand*{\theoremref}[1]{\hyperref[#1]{Theorem~\ref*{#1}}}
\newcommand*{\corollaryref}[1]{\hyperref[#1]{Corollary~\ref*{#1}}}
\newcommand*{\propositionref}[1]{\hyperref[#1]{Proposition~\ref*{#1}}}
\newcommand*{\lemmaref}[1]{\hyperref[#1]{Lemma~\ref*{#1}}}
\newcommand*{\definitionref}[1]{\hyperref[#1]{Definition~\ref*{#1}}}
\newcommand*{\exampleref}[1]{\hyperref[#1]{Example~\ref*{#1}}}
\newcommand*{\remarkref}[1]{\hyperref[#1]{Remark~\ref*{#1}}}
\newcommand*{\remarksref}[1]{\hyperref[#1]{Remarks~\ref*{#1}}}
\newcommand*{\sectionref}[1]{\hyperref[#1]{Section~\ref*{#1}}}
\makeatletter
\renewcommand{\eqref}[1]{\hyperref[#1]{\textup{\tagform@{\ref*{#1}}}}}
\makeatother

\newcommand*{\N}{\ensuremath{\mathbb{N}}}
\newcommand*{\Z}{\ensuremath{\mathbb{Z}}}

\newcommand*{\R}{\ensuremath{\mathbb{R}}}
\newcommand*{\C}{\ensuremath{\mathbb{C}}}

\newcommand*{\D}{\ensuremath{\mathrm{d}}}
\newcommand*{\E}{\ensuremath{\mathrm{e}}}
\newcommand*{\I}{\ensuremath{\mathrm{i}}}

\newcommand*{\CC}{\ensuremath{\mathrm{C}}}
\newcommand*{\GL}{\ensuremath{\mathrm{GL}}}
\newcommand*{\SO}{\ensuremath{\mathrm{SO}}}

\newcommand*{\Spin}{\ensuremath{\mathrm{Spin}}}
\newcommand*{\vol}{\ensuremath{\mathrm{vol}}}

\newcommand*{\sections}{\ensuremath{\mathrm{\Gamma}}}
\newcommand*{\altforms}{\ensuremath{\mathrm{\Lambda}}}
\newcommand*{\forms}{\ensuremath{\mathrm{\Omega}}}
\newcommand*{\interior}{\ensuremath{\mathbin\lrcorner}}

\newcommand*{\wolog}{w.l.o.g.}

\let\Re=\relax
\let\Im=\relax
\DeclareMathOperator{\Re}{Re}
\DeclareMathOperator{\Im}{Im}
\DeclareMathOperator{\id}{id}
\DeclareMathOperator{\sgn}{sgn}

\DeclareMathOperator{\ind}{ind}

\DeclareMathOperator{\Ch}{\mathbf{ch}}
\DeclareMathOperator{\Ahat}{\hat{\vectorsym{A}}}

\newstarcommand*{\defeq}{\ensuremath{\phantom{\mathrel{\mathop:}}=}}{\ensuremath{\mathrel{\mathop:}=}}
\newstarcommand*{\eqdef}{\ensuremath{=\phantom{\mathrel{\mathop:}}}}{\ensuremath{=\mathrel{\mathop:}}}
\newcommand*{\dash}[1]{\nobreakdash#1\hspace{0pt}}
\newcommand*{\normal}[2][]{\ensuremath{\nu_{#1}#2}}

\makeatletter
\newstarcommand*{\norm}{\@norm{{}\cdot{}}}{\@norm}
\DeclarePairedDelimiter{\@norm}{\lVert}{\rVert}
\newstarcommand*{\abs}{\@abs{{}\cdot{}}}{\@abs}
\DeclarePairedDelimiter{\@abs}{\lvert}{\rvert}
\newstarcommand*{\inner}{\@inner{{}\cdot{},{}\cdot{}}}{\@inner}
\DeclarePairedDelimiter{\@inner}{\langle}{\rangle}
\newstarcommand*{\set}{\emptyset}{\@set}
\DeclarePairedDelimiter{\@set}{\lbrace}{\rbrace}
\newstarcommand*{\paren}{\@paren{{}\dots{}}}{\@paren}
\DeclarePairedDelimiter{\@paren}{\lparen}{\rparen}
\newstarcommand*{\bracket}{\@bracket{{}\dots{}}}{\@bracket}
\DeclarePairedDelimiter{\@bracket}{\lbrack}{\rbrack}
\makeatother

\newcommand{\hodgestar}{\ensuremath{\mathord\ast}}
\newcommand{\Bev}{\ensuremath{B^{\text{\textup{ev}}}}}
\DeclareMathOperator{\Pf}{Pf}
\newcommand*{\End}{\ensuremath{\mathrm{End}}}

\newcommand*{\CP}{\ensuremath{\mathbb{CP}}}
\newcommand{\conjugate}[1]{\ensuremath{\overline{#1\!}\,}}
\DeclareMathOperator{\SF}{sf}

\title{Deformations of Asymptotically Cylindrical Cayley Submanifolds}
\author{Matthias Ohst \\[1.5ex]
DPMMS, University of Cambridge, United Kingdom \\
\href{mailto:M.Ohst@dpmms.cam.ac.uk}{\nolinkurl{M.Ohst@dpmms.cam.ac.uk}}}
\date{}

\hypersetup{
  pdfpagemode = UseNone,
  hidelinks,
}

\hypersetup{
  pdftitle    = {Deformations of Asymptotically Cylindrical Cayley Submanifolds},
  pdfauthor   = {Matthias Ohst},
  pdfsubject  = {},
  pdfkeywords = {},
}

\begin{document}

\maketitle

\begin{abstract}
  We study the deformations of an asymptotically cylindrical Cayley submanifold inside an asymptotically cylindrical $\Spin(7)$-manifold. We prove an index formula for the operator of Dirac type that arises as the linearisation of the deformation map and show that if the $\Spin(7)$-structure is generic, then there are no obstructions, and hence the moduli space is a smooth finite-dimensional manifold whose dimension is equal to the index of the operator of Dirac type. We further construct examples of asymptotically cylindrical Cayley submanifolds inside the asymptotically cylindrical Riemannian $8$\dash-manifolds with holonomy $\Spin(7)$ constructed by Kovalev.
\end{abstract}

\section{Introduction}
\label{sec:Introduction}

Cayley submanifolds of~$\R^8$ were introduced by Harvey and Lawson \cite{HL82} as an instance of calibrated submanifolds, extending the volume-minimising properties of complex submanifolds in Kähler manifolds. Other classes of calibrated submanifolds given in \cite{HL82} are the special Lagrangian submanifolds of~$\C^n$ and the associative and coassociative submanifolds of~$\R^7$. More generally, Cayley submanifolds are $4$\dash-dimensional submanifolds which may be defined in an $8$\dash-manifold~$M$ equipped with a certain differential $4$\dash-form~$\Phi$ invariant at each point under the spin representation of $\Spin(7)$. The latter representation identifies $\Spin(7)$ as a subgroup of $\SO(8)$, and a $\Spin(7)$-structure determined by~$\Phi$ induces a Riemannian metric and orientation on~$M$. See \sectionref{subsec:preliminaries-spin7} for details.

Calibrated submanifolds often arise in Riemannian manifolds with reduced holonomy. In particular, Cayley submanifolds in an $8$\dash-manifold are calibrated and minimal whenever the respective ``$\Spin(7)$-structure''~$\Phi$ is closed. In that case, the holonomy of the Riemannian metric induced by~$\Phi$ reduces to a subgroup of $\Spin(7)$; in particular, the metric then is Ricci-flat. The first examples of closed Riemannian $8$\dash-manifolds with holonomy $\Spin(7)$ were constructed by Joyce \cite{Joy96}. He also provided examples of closed Cayley submanifolds inside these manifolds \cite{Joy00}.

McLean \cite{McL98} studied the deformations of closed calibrated submanifolds for the calibrations introduced in \cite{HL82}. He showed, among other results, that the deformation problem in each case is elliptic or overdetermined elliptic and described the respective finite-dimensional Zariski tangent spaces. In the case of closed Cayley submanifolds, the deformations may in general be obstructed, and the Zariski tangent space can be given in terms of harmonic spinors of a certain twisted Dirac operator.

Later, other authors extended McLean's results to larger classes of submanifolds, including, in the special Lagrangian and coassociative cases, asymptotically cylindrical submanifolds. See, respectively, Joyce and Salur \cite{JS05} and Salur and Todd \cite{ST10}. The corresponding moduli spaces of deformations were shown to be finite-dimensional smooth manifolds.

In this article, we study the deformations of asymptotically cylindrical Cayley submanifolds. The first examples of asymptotically cylindrical Riemannian $8$\dash-manifolds with holonomy~$\Spin(7)$ were constructed by Kovalev \cite{Kov13}. Among other results, we will construct examples of asymptotically cylindrical Cayley submanifolds inside these $\Spin(7)$-manifolds (see \sectionref{sec:examples}). These Cayley submanifolds are constructed as the fixed-point set of an involution that preserves the $\Spin(7)$-structure (using a result of Joyce \cite{Joy00}).

We further show that if the asymptotically cylindrical $\Spin(7)$-structure is generic, then asymptotically cylindrical Cayley submanifolds form a smooth finite-dimensional moduli space (see \sectionref{sec:varying-spin7}). The dimension of the moduli space under such a genericity assumption is equal to the index of a certain operator of Dirac type. In \sectionref{sec:index-formula}, we prove an index formula for this operator using the Atiyah--Patodi--Singer Index Theorem (\theoremref{thm:main-index-formula}). We also prove an alternative index formula that involves the spectral flow between two operators (\propositionref{prop:index-formula-spectral-flow}).

The cross-section at infinity of the examples of asymptotically cylindrical Cayley submanifolds inside the $\Spin(7)$-manifolds of~\cite{Kov13} which we construct in \sectionref{sec:examples} have certain properties that help to simplify the general index formula. In \sectionref{subsec:additional-assumptions}, we prove simplified index formulae (which only involve topological quantities) under such special assumptions.

If the holonomy group of the ambient $8$\dash-manifold is a proper subgroup of~$\Spin(7)$, then examples of asymptotically cylindrical Cayley submanifolds can be constructed from submanifolds that are calibrated with respect to another calibration. In \sectionref{sec:relation-other-calibrations}, we investigate the deformations of such Cayley submanifolds. In particular, we show that any deformation of such a Cayley submanifold as a Cayley submanifold must again be of that form. We further simplify the index formulae in these cases.

\paragraph{Acknowledgements.} I am grateful to my Ph.D. supervisor, Alexei Kovalev, for his support and guidance.

\section{Preliminaries}
\label{sec:preliminaries}

We begin by reviewing $\Spin(7)$-structures on $8$\dash-manifolds and asymptotically cylindrical Riemannian manifolds. We also review the Atiyah--Patodi--Singer Index Theorem and the relative Euler class, which we will use in the index formulae. We finish this section by proving an extension of Harvey and Lawson's \cite{HL82} volume-minimising property of calibrated submanifolds to asymptotically cylindrical calibrated submanifolds.

\subsection{Spin(7)-Structures}
\label{subsec:preliminaries-spin7}

Here we recall some basic facts about $\Spin(7)$-structures on $8$\dash-manifolds and Cayley submanifolds (see, for example, \cite{HL82}, \cite{Joy00}).

Let $(x_1, \dotsc, x_8)$ be coordinates on~$\R^8$, and write $\D \vectorsym{x}_{i \dots j}$ for $\D x_i \wedge \dotsb \wedge \D x_j$. Define a $4$\dash-form $\Phi_0$ on~$\R^8$ by
\begin{equation}
  \begin{split}
    \Phi_0 &\defeq \D \vectorsym{x}_{1234} + (\D \vectorsym{x}_{12} - \D \vectorsym{x}_{34}) \wedge (\D \vectorsym{x}_{56} - \D \vectorsym{x}_{78}) \\
    &\phantom{{}\defeq{}} {}+ (\D \vectorsym{x}_{13} + \D \vectorsym{x}_{24}) \wedge (\D \vectorsym{x}_{57} + \D \vectorsym{x}_{68}) \\
    &\phantom{{}\defeq{}} {}+ (\D \vectorsym{x}_{14} - \D \vectorsym{x}_{23}) \wedge (\D \vectorsym{x}_{58} - \D \vectorsym{x}_{67}) + \D \vectorsym{x}_{5678} \, \text{.}
  \end{split} \label{eq:def-Phi-0}
\end{equation}
The subgroup of $\GL(8, \R)$ preserving~$\Phi_0$ is isomorphic to $\Spin(7)$, viewed as a subgroup of $\SO(8)$. Note that $\Phi_0$ is self-dual.

Let $M$ be an $8$\dash-manifold. Suppose that there is a $4$\dash-form~$\Phi$ on~$M$ such that for each $x \in M$ there is a linear isomorphism $i_x \colon T_x M \to \R^8$ with $(i_x)^\ast(\Phi_0) = \Phi_x$ (in a neighbourhood of each point, this can be chosen to depend smoothly on~$x$). Then $\Phi$ induces a $\Spin(7)$\dash-structure on~$M$. Conversely, if $M$ has a $\Spin(7)$\dash-structure, then there is such a $4$\dash-form~$\Phi$. Via such an identification $i_x \colon T_x M \to \R^8$ of~$\Phi_x$ with~$\Phi_0$, the metric~$g_0$ of~$\R^8$ induces a metric $(i_x)^\ast(g_0)$ on~$T_x M$. Since $\Spin(7) \subseteq \SO(8)$, this metric is independent of the chosen identification, and we get a well-defined Riemannian metric~$g = g(\Phi)$ and orientation on~$M$. By abuse of notation, we will refer to the $4$\dash-form~$\Phi$ as a $\Spin(7)$\dash-structure. The $\Spin(7)$\dash-structure is called \emph{torsion-free} if $\nabla \Phi = 0$, where $\nabla$ is the Levi-Civita connection of~$(M, g)$. This is equivalent to $\D \Phi = 0$ \cite[Theorem~5.3]{Fer86}.

If $M$ is an $8$\dash-manifold, then there exists a $\Spin(7)$-structure on~$M$ if and only if $M$ is orientable and spin and
\begin{equation}
  p_1(M)^2 - 4 p_2(M) + 8 \chi(M) = 0
\end{equation}
for some orientation of~$M$ \cite[Theorem~10.7 in Chapter~IV]{LM89}, where $p_i(M)$ is the $i$\dash-th Pontryagin class and $\chi(M)$ is the Euler characteristic of~$M$.

Now let $M$ be an $8$\dash-manifold with a $\Spin(7)$-structure~$\Phi$. Then we have pointwise orthogonal splittings \cite[Lemmas~3.1 and 3.3]{Fer86}
\begin{equation}
  \begin{split}
    \altforms^2 M &= \altforms^2_7 M \oplus \altforms^2_{21} M \quad \text{and} \\
    \altforms^4 M &= \altforms^4_1 M \oplus \altforms^4_7 M \oplus \altforms^4_{27} M \oplus \altforms^4_{35} M \, \text{.}
  \end{split} \label{eq:forms-splitting}
\end{equation}
Here $\altforms^p M \defeq \altforms^p T^\ast M$, and $\altforms^k_\ell M$ corresponds to an irreducible representation of~$\Spin(7)$ of dimension~$\ell$. Furthermore, $M$ possesses a $2$\dash-fold cross product $T M \times T M \to \altforms^2_7 M$,
\begin{equation}
  v \times w \defeq 2 \pi_7(v^\flat \wedge w^\flat) = \tfrac{1}{2} (v^\flat \wedge w^\flat - \mathord\ast (v^\flat \wedge w^\flat \wedge \Phi)) \label{eq:def-cross-2}
\end{equation}
and a $3$\dash-fold cross product $T M \times T M \times T M \to T M$,
\begin{equation}
  u \times v \times w \defeq (u \interior (v \interior (w \interior \Phi)))^\sharp \, \text{.} \label{eq:def-cross-3}
\end{equation}
They satisfy $\abs{v \times w} = \abs{v \wedge w}$, $\abs{u \times v \times w} = \abs{u \wedge v \wedge w}$, and
\begin{equation}
  h(a \times b, c \times d) = - \Phi(a, b, c, d) + g(a, c) g(b, d) - g(a, d) g(b, c) \, \text{,} \label{eq:inner-cross-2}
\end{equation}
where $h$ is the induced metric on~$\altforms^2_7 M$. There is also a vector-valued $4$\dash-form $\tau \in \forms^4(M, \altforms^2_7 M)$ (also called the $4$\dash-fold cross product),
\begin{equation}
  \tau(a, b, c, d) \defeq - a \times (b \times c \times d) + g(a, b) (c \times d) + g(a, c) (d \times b) + g(a, d) (b \times c) \, \text{,} \label{eq:def-tau}
\end{equation}
which satisfies
\begin{equation}
  h(\tau, v \times w) = w^\flat \wedge (v \interior \Phi) - v^\flat \wedge (w \interior \Phi) \, \text{.} \label{eq:inner-tau}
\end{equation}
This formula can be checked by using the invariance properties of the cross products and checking it for $v = e_1$, $w = e_2$, where $(e_1, \dotsc, e_8)$ is a $\Spin(7)$\dash-frame (see below for the definition). Note that
\begin{equation}
  w^\flat \wedge (v \interior \Phi) - v^\flat \wedge (w \interior \Phi) \in \altforms^4_7 M \, \text{,} \label{eq:forms-4-7}
\end{equation}
which follows from \cite[page~548]{Bry87}.

Let $x \in M$, and let $(e_1, \dotsc, e_8)$ be a basis of~$T_x M$. We call $(e_1, \dotsc, e_8)$ a \emph{$\Spin(7)$\dash-frame} if
\begin{equation}
  \begin{split}
    \Phi &= e^{1234} + e^{1256} - e^{1278} + e^{1357} + e^{1368} + e^{1458} - e^{1467} \\
         &\phantom{{}={}} {}- e^{2358} + e^{2367} + e^{2457} + e^{2468} - e^{3456} + e^{3478} + e^{5678} \, \text{,}
  \end{split} \label{eq:phi-spin7-frame}
\end{equation}
where $(e^1, \dotsc, e^8)$ is the dual coframe. Note that if $(e_1, \dotsc, e_8)$ is a $\Spin(7)$-frame, then it is an orthonormal frame since $\Spin(7) \subseteq \SO(8)$. Furthermore, $e_i \times e_j = \pm e_k \times e_\ell$ if and only if $\Phi(e_i, e_j, e_k, e_\ell) = \mp 1$ for $i, j, k, \ell \in \set{1, \dotsc, 8}$ different by~\eqref{eq:inner-cross-2}. So \eqref{eq:phi-spin7-frame} shows that
\begin{equation}
  \begin{split}
    e_1 \times e_5 &= \phantom{+} e_2 \times e_6 = \phantom{+} e_3 \times e_7 = \phantom{+} e_4 \times e_8 \, \text{,} \\
    e_1 \times e_6 &= - e_2 \times e_5 = \phantom{+} e_3 \times e_8 = - e_4 \times e_7 \, \text{,} \\
    e_1 \times e_7 &= - e_2 \times e_8 = - e_3 \times e_5 = \phantom{+} e_4 \times e_6 \, \text{,} \\
    e_1 \times e_8 &= \phantom{+} e_2 \times e_7 = - e_3 \times e_6 = - e_4 \times e_5 \, \text{.}
  \end{split} \label{eq:cross-product-table}
\end{equation}

If $e_1, e_2, e_3 \in T_x M$ are orthogonal unit vectors and $e_5 \in T_x M$ is a unit vector that is orthogonal to~$e_1, e_2, e_3, e_1 \times e_2 \times e_3$, then there are (uniquely determined) $e_4, e_6, e_7, e_8 \in T_x M$ such that $(e_1, \dotsc, e_8)$ is a $\Spin(7)$\dash-frame, namely $e_4 = - e_1 \times e_2 \times e_3$, $e_6 = - e_1 \times e_2 \times e_5$, $e_7 = - e_1 \times e_3 \times e_5$, and $e_8 = e_2 \times e_3 \times e_5$.

We have $\Phi_x \vert_V \le \vol_V$ for all~$x \in M$ and every oriented $4$\dash-dimensional subspace~$V$ of~$T_x M$, where $\vol_V$ is the volume form (induced by the metric~$g$ and the orientation on~$V$) and $\varphi_x \vert_V \le \vol_V$ means that $\varphi_x \vert_V = \lambda \, \vol_V$ with~$\lambda \le 1$. An orientable $4$\dash-dimensional submanifold~$X$ of~$M$ is called \emph{Cayley} if $\Phi \vert_X = \vol_X$ for some orientation of~$X$. This is equivalent to~$\tau \vert_X = 0$ \cite[Corollary~1.29 in Chapter~IV]{HL82}. If the $\Spin(7)$-structure~$\Phi$ is torsion-free, then $\Phi$ is a calibration on~$M$, and Cayley submanifolds are minimal submanifolds \cite[Theorem~4.2 in Chapter~II]{HL82}.

Now suppose that $X$ is a Cayley submanifold of~$M$. Then $\altforms^2_- X$ is isomorphic to a subbundle of~$\altforms^2_7 M \vert_X$ \cite[Section~6]{McL98} via the embedding
\begin{equation}
  \altforms^2_- X \to \altforms^2_7 M \vert_X \, \text{,} \quad \alpha \mapsto 2 \pi_7(\alpha) = \tfrac{1}{2} (\alpha - \mathord\ast (\alpha \wedge \Phi)) \, \text{,} \label{eq:subbundle-2-7}
\end{equation}
where we extend $\alpha \in \altforms^2_- X$ to $\altforms^2 M \vert_X$ by $v \interior \alpha = 0$ for all~$v \in \normal[M]{X}$. Let $E$ denote the orthogonal complement of~$\altforms^2_- X$ in~$\altforms^2_7 M \vert_X$. So
\begin{equation}
  \altforms^2_7 M \vert_X \cong \altforms^2_- X \oplus E \, \text{,} \label{eq:splitting-2-7}
\end{equation}
and $E$ has rank~$4$. Furthermore,
\begin{equation}
  E = \set{\alpha \in \altforms^2_7 M \vert_X \colon \alpha \vert_{T X} = 0} \, \text{.} \label{eq:def-E}
\end{equation}
The cross products restrict to
\begin{equation}
  T X \times T X \to \altforms^2_- X \, \text{,} \quad T X \times \normal[M]{X} \to E \, \text{,} \quad \normal[M]{X} \times \normal[M]{X} \to \altforms^2_- X \label{eq:cross-2-restriction}
\end{equation}
and
\begin{equation}
  \begin{split}
    &\phantom{T X \times \normal[M]{X} \times \normal[M]{X} \to T X \, \text{,} \quad}\mathllap{T X \times T X \times T X \to T X \, \text{,} \quad} \phantom{\normal[M]{X} \times \normal[M]{X} \times \normal[M]{X} \to \normal[M]{X}} \mathllap{T X \times T X \times \normal[M]{X} \to \normal[M]{X}} \, \text{,} \\
    &T X \times \normal[M]{X} \times \normal[M]{X} \to T X \, \text{,} \quad \normal[M]{X} \times \normal[M]{X} \times \normal[M]{X} \to \normal[M]{X} \, \text{.}
  \end{split} \label{eq:cross-3-restriction}
\end{equation}

\subsection{Asymptotically Cylindrical Manifolds}
\label{subsec:preliminaries-acyl}

Here we give the definition of asymptotically cylindrical manifolds and present the Fredholm theory of asymptotically cylindrical linear elliptic operators, taken from \cite{LM85}.

\begin{definition} \label{def:ACyl}
  Let $(M, g)$ be a complete Riemannian manifold. Then $(M, g)$ is called \emph{asymptotically cylindrical} if there are a compact subset $K \subseteq M$, a closed Riemannian manifold $(N, h)$, and a diffeomorphism $\Psi \colon (0, \infty) \times N \to M \setminus K$ such that
  \begin{equation*}
    \abs{\nabla_\infty^k(\Psi^\ast(g) - g_\infty)} = \mathcal{O}(\E^{- \delta t}) \quad \text{for all $k \in \N$}
  \end{equation*}
  for some $\delta > 0$, where $t$ denotes the projection onto the $(0, \infty)$\dash-factor and $\nabla_\infty$ is the Levi-Civita connection of $(0, \infty) \times N$ with respect to the product metric $g_\infty = \D t^2 + h$.
  
  A submanifold~$X$ of~$M$ is called \emph{asymptotically cylindrical} if $(X, g \vert_X)$ is complete and there are a compact subset $K^\prime \subseteq X$, a closed submanifold $Y$ of~$N$, and a section~$v$ of the normal bundle of $(R, \infty) \times Y$ in $(R, \infty) \times N$ (for some $R \ge 0$) with
  \begin{equation*}
    \abs{\nabla_\infty^k v} = \mathcal{O}(\E^{- \delta^\prime t}) \quad \text{for all $k \in \N$}
  \end{equation*}
  for some $\delta^\prime > 0$ such that
  \begin{equation*}
    (R, \infty) \times Y \to M \, \text{,} \quad (t, x) \mapsto \Psi(\exp_{(t, x)}(v(t, x)))
  \end{equation*}
  is a diffeomorphism onto $X \setminus K^\prime$, where $\exp$ is the exponential map with respect to the metric~$g_\infty$.
\end{definition}

\begin{theorem}[{\cite[Theorem~6.2]{LM85}}]
  Let $(M, g)$ be an asymptotically cylindrical manifold, and let $P \colon \sections(E) \to \sections(F)$ be an asymptotically cylindrical linear ellip\-tic differential operator of order~$k$ that is asymptotic to the translation-invariant operator $P_\infty \colon \sections(E_\infty) \to \sections(F_\infty)$.
  
  Then $P$ extends to a bounded linear map $P^{\ell, \alpha}_\lambda \colon \CC^{\ell + k, \alpha}_\lambda(E) \to \CC^{\ell, \alpha}_\lambda(F)$ for all $\ell \in \N$, $\alpha \in (0, 1)$, and $\lambda \in \R$. For $\lambda \in \R$, let $d_P(\lambda)$ be the complex dimension of the space spanned by the solutions $s \in \sections(E_\infty)$ of $P_\infty s = 0$ on the cylinder $\R \times N$ such that $\E^{(\lambda + \I \gamma) t} s$ is polynomial in~$t$ for some $\gamma \in \R$, and let
  \begin{equation}
    \mathcal{D}_P \defeq \set{\lambda \in \R \colon d_P(\lambda) \ne 0} \, \text{.}
  \end{equation}
  Then $\mathcal{D}_P$ is a discrete set such that $P^{\ell, \alpha}_\lambda \colon \CC^{\ell + k, \alpha}_\lambda(E) \to \CC^{\ell, \alpha}_\lambda(F)$ is Fredholm if and only if $\lambda \notin \mathcal{D}_P$. For $\lambda \in \R \setminus \mathcal{D}_P$, let $\ind_\lambda P \defeq \ind P^{\ell, \alpha}_\lambda$, which is independent of $\ell$ and $\alpha$ and hence well-defined. Then
  \begin{equation}
    \ind_\lambda P - \ind_\delta P = \sum_{\substack{\gamma \in \mathcal{D}_P \\ \lambda < \gamma < \delta}} d_P(\gamma)
  \end{equation}
  for all $\lambda, \delta \in \R \setminus \mathcal{D}_P$ with $\lambda < \delta$.
\end{theorem}

\begin{lemma}
  Let $M$ be an asymptotically cylindrical manifold, and let $P \colon \sections(E) \to \sections(F)$ and $Q \colon \sections(E) \to \sections(F)$ be asymptotically cylindrical linear elliptic differential operators of the same order. If $P$ and $Q$ have the same symbol and are asymptotic to the same translation-invariant operator, then $\ind_\lambda P = \ind_\lambda Q$ for all $\lambda \notin \mathcal{D}_P = \mathcal{D}_Q$.
\end{lemma}

\subsection{Atiyah--Patodi--Singer Index Theorem}
\label{subsec:preliminaries-aps-theorem}

The main tool to compute the index of the operator of Dirac type that arises as the linearisation of the deformation map will be the Atiyah--Patodi--Singer Index Theorem \cite{APS75a}. Here we recall the Atiyah--Patodi--Singer Index Theorem and also the signature theorem for asymptotically cylindrical manifolds.

\begin{definition}[{\cite{APS75a}}]
  Let $M$ be an odd-dimensional closed manifold, and let $P \colon \sections(E) \to \sections(F)$ be a first-order linear elliptic differential operator. Define
  \begin{equation}
    \eta_P(z) \defeq \sum_{\lambda \ne 0} (\sgn \lambda) \abs{\lambda}^{- z}
  \end{equation}
  for $z \in \C$ with $\Re z$ large, where $\lambda$ runs over all eigenvalues of~$P$ (counted with multiplicity). This function has a meromorphic extension to~$\C$, and $z = 0$ is not a pole. The \emph{$\eta$\dash-invariant} of~$P$ is defined by $\eta(P) \defeq \eta_P(0)$.
\end{definition}

\begin{theorem}[{Atiyah--Patodi--Singer Index Theorem \cite[(4$\cdot$3)]{APS75a} (see also \cite[Theorem~9.1]{Mel93})}] \label{thm:atiyah-patodi-singer}
  Let $M$ be an $n$\dash-dimensional asymptotically cylindrical manifold with cross-section~$N$, let $\mathbb{S}^+ \otimes F$ and $\mathbb{S}^- \otimes F$ be positive and negative (twisted) spinor bundles, respectively, and let $D \colon \sections(\mathbb{S}^+ \otimes F) \to \sections(\mathbb{S}^- \otimes F)$ be the positive (twisted) Dirac operator. Then
  \begin{equation}
    \ind_\lambda D = \int_M (\Ahat(M) \Ch(F))_n - \frac{\dim \ker \tilde{D} + \eta(\tilde{D})}{2} \, \text{,} \label{eq:atiyah-patodi-singer}
  \end{equation}
  where $\lambda > 0$ is such that $(0, \lambda] \cap \mathcal{D}_D = \emptyset$. Here $\tilde{D} \colon \sections((\mathbb{S}^+ \otimes F) \vert_N) \to \sections((\mathbb{S}^+ \otimes F) \vert_N)$ is defined by the equation
  \begin{equation*}
    D s = u \cdot (\nabla_u s + \tilde{D} s)
  \end{equation*}
  for all $s \in \sections(\mathbb{S}^+ \otimes F)$, where $u \defeq - \frac{\partial}{\partial t}$, and $\eta(\tilde{D})$ is the $\eta$\dash-invariant of~$\tilde{D}$.
\end{theorem}

Note that the index of a negative (twisted) Dirac operator $D \colon \sections(\mathbb{S}^- \otimes F) \to \sections(\mathbb{S}^+ \otimes F)$ is therefore given by
\begin{equation}
  \ind_\lambda D = - \int_M (\Ahat(M) \Ch(F))_n - \frac{\dim \ker \tilde{D} - \eta(\tilde{D})}{2} \, \text{,} \label{eq:atiyah-patodi-singer-negative}
\end{equation}
where $\tilde{D} \colon \sections((\mathbb{S}^- \otimes F) \vert_N) \to \sections((\mathbb{S}^- \otimes F) \vert_N)$ is defined by the equation
\begin{equation*}
  D s = - u \cdot (- \nabla_u s + \tilde{D} s) \, \text{.}
\end{equation*}

\begin{theorem}[{\cite[Theorem~(4$\cdot$14)]{APS75a} (see also \cite[Theorem~9.4]{Mel93})}] \label{thm:atiyah-patodi-singer-signature}
  Let $M$ be an $n$\dash-dimensional asymptotically cylindrical manifold with cross-section~$N$. Suppose that $n = 4 k$. Then
  \begin{equation}
    \sigma(M) = \int_M L_k(p) - \eta(\Bev) \, \text{,}
  \end{equation}
  where
  \begin{compactenum}[(i)]
    \item $\sigma(M)$~is the signature of~$M$ (the signature of the non-degenerate quadratic form induced by the cup-product on the image of $H^{2 k}_{\text{\textup{cs}}}(M)$ in $H^{2 k}(M)$),
    \item $L_k(p)$~is the $k$\dash-th Hirzebruch $L$\dash-polynomial in the Pontryagin forms, and
    \item \label{thm:atiyah-patodi-singer-signature-Bev} $\Bev \colon \forms^{\text{\textup{ev}}}(N) \to \forms^{\text{\textup{ev}}}(N)$ is defined by
    \begin{equation*}
      \Bev(\varphi) \defeq (-1)^{k + p} (\hodgestar \D - \D \hodgestar) \varphi
    \end{equation*}
    for $\varphi \in \forms^{2 p}(N)$.
  \end{compactenum}
\end{theorem}

Note that this is the negative of~$\Bev$ in \cite{APS75a} since we use the opposite orientation on~$N$.

\begin{lemma} \label{lem:bev-zero}
  Let $M$ be an oriented, odd-dimensional Riemannian manifold, and let $\Bev \colon \forms^{\text{\textup{ev}}}(M) \to \forms^{\text{\textup{ev}}}(M)$ be defined as in \eqref{thm:atiyah-patodi-singer-signature-Bev} of \theoremref{thm:atiyah-patodi-singer-signature}. Suppose that $M$ has an orientation-reversing isometry. Then $\eta(\Bev) = 0$.
\end{lemma}

This happens, in particular, if $M = S^1 \times N$ for some oriented Riemannian manifold~$N$.

\subsection{Relative Euler Class and Generalised Gauss--Bonnet--Chern Theorem}
\label{subsec:preliminaries-euler-class}

In the index formulae that we prove, we will use the relative Euler class. Here we define the relative Euler class and present the general Gauss--Bonnet--Chern Theorem for the computation of the relative Euler class (see \cite{Sha73}).

\begin{definition}
  Let $M$ be a manifold, let $N$ be a submanifold of~$M$, let $E$ be an oriented vector bundle over~$M$ of rank~$n$, and let $s \in \sections(E \vert_N)$ be a non-vanishing section of~$E$ over~$N$. The \emph{relative Euler class} $e(E, s) \in H^n(M, N)$ is defined as follows. Let $E_0 \subseteq E$ be the complement of the zero section, let $S \defeq s(N)$ (note that $S \subseteq E_0$ since $s$ is non-vanishing), and let $u \in H^n(E, E_0)$ be the orientation class. Furthermore, let $p \colon (E, S) \to (M, N)$ denote the projection (note that $p$ is a homotopy equivalence), and let $i \colon (E, S) \to (E, E_0)$ denote the inclusion. Then $e(E, s) \defeq (p^\ast)^{-1} i^\ast u$.
\end{definition}

If $M$ is a connected, oriented, $n$\dash-dimensional asymptotically cylindrical manifold with cross-section~$N$, $E$~is an oriented vector bundle of rank~$n$ over~$M$, and $s \in \sections(E \vert_N)$ is a non-vanishing section, then $H^n(M, N) \cong \Z$, and we can interpret the relative Euler class $e(E, s)$ as an integer.

\begin{theorem}[``Generalised Gauss--Bonnet--Chern Theorem'' (cf. {\cite[\S 3]{Sha73}})] \label{thm:gauss-bonnet-chern}
  Let $M$ be a connected, oriented, $n$\dash-dimensional asymptotically cylindrical manifold with cross-section~$N$, let $E$ be an oriented, asymptotically cylindrical vector bundle over~$M$ of rank~$n$, let $\nabla$ be an asymptotically cylindrical metric connection on~$E$, and let $s \in \sections(E \vert_N)$ be a non-vanishing section with point-wise norm~$1$. Suppose that $n =  2 k$ is even. Then
  \begin{equation}
    e(E, s) = \int_M \frac{1}{(2 \pi)^k} \Pf(F_\nabla) + \int_N \Theta(F_\nabla \vert_N, \nabla s) \, \text{,} \label{eq:gauss-bonnet-chern}
  \end{equation}
  where $F_\nabla \in \forms^2(M, \End(E))$ is the curvature of~$\nabla$, $\Pf$~denotes the Pfaffian on $\End(E)$, and $\Theta(F_\nabla \vert_N, \nabla s) \in \forms^{n - 1}(N)$ is a differential form that
  \begin{compactenum}[(i)]
    \item depends only on $F_\nabla \vert_N \in \forms^2(N, \End(E \vert_N))$ and $\nabla s \in \forms^1(\partial M, E \vert_N)$,
    \item satisfies $\Theta(F_\nabla \vert_N, \nabla s) = 0$ if $\nabla s = 0$, and
    \item \label{thm:gauss-bonnet-chern-part-non-vanishing} satisfies $\Theta(F_\nabla \vert_N, \nabla s) = 0$ if $v \interior F_\nabla = 0$ and $\nabla_v s = 0$ for some non-vanishing $v \in \sections(T N)$.
  \end{compactenum}
\end{theorem}

\begin{proof}
  A proof of an analogous statement for compact manifolds with boundary works almost literally like in \cite{Sha73}, except that we use $s$ instead of the exterior normal vector field. The version for asymptotically cylindrical manifolds can then be deduced by going to the limit $t \to \infty$.
\end{proof}

\subsection{Volume Minimising Property}
\label{subsec:volume-minimising-property}

Harvey and Lawson \cite{HL82} proved that closed calibrated submanifolds are volume-minimising in their homology class. In fact, they proved this for compactly supported deformations. Here we present an extension of this property to asymptotically cylindrical calibrated submanifolds, which we will apply in \sectionref{sec:relation-other-calibrations}.

\begin{proposition}[{cf.~\cite[Theorem~4.2 in Chapter~II]{HL82}}] \label{prop:minimal}
  Let $M$ be an asymptotically cylindrical manifold with cross-section~$N$, let $\varphi \in \forms^n(M)$ be an asymptotically cylindrical calibration on~$M$, let $X$ be an asymptotically cylindrical calibrated submanifold of~$M$, and let $Y$ be an asymptotically cylindrical, oriented, $n$\dash-dimensional submanifold of~$M$ with the same asymptotic limit as~$X$ which lies in the same relative homology class as~$X$ in $H_n(M, N)$. Furthermore, fix a diffeomorphism $M \setminus K \cong (0, \infty) \times N$ as in \definitionref{def:ACyl}, where $K$ is a compact subset of~$M$, and let $t$ denote the projection onto the $(0, \infty)$\dash-factor.
  
  Then
  \begin{equation}
    \lim_{T \to \infty} \bigl(\vol(\set{y \in Y \colon t \le T}) - \vol(\set{x \in X \colon t \le T})\bigr) \ge 0 \, \text{,}
  \end{equation}
  and equality holds if and only if $Y$ is calibrated.
\end{proposition}

\begin{proof}
  Let $T > 0$ be large enough, and let $X_T \defeq \set{x \in X \colon t \le T}$ and $Y_T \defeq \set{y \in Y \colon t \le T}$. Since $X$ and $Y$ lie in the same relative homology class in $H_n(M, N)$, there are an $(n + 1)$-chain~$S_T$ in~$M$ and an $n$\dash-chain~$Z_T$ in~$N$ such that $Y_T - X_T = \partial S_T + Z_T$. So
  \begin{equation*}
    \int_{Y_T} \varphi - \int_{X_T} \varphi = \int_{\partial S_T} \varphi + \int_{Z_T} \varphi = \int_{S_T} \D \varphi + \int_{Z_T} \varphi = \int_{Z_T} \varphi
  \end{equation*}
  by Stokes' Theorem since $\D \varphi = 0$ as $\varphi$ is a calibration. Furthermore,
  \begin{equation*}
    \lim_{T \to \infty} \int_{Z_T} \varphi = 0
  \end{equation*}
  since $X$ and $Y$ have the same asymptotic limit. Hence
  \begin{align*}
    \lim_{T \to \infty} \bigl(\vol(X_T) - \vol(Y_T)\bigr) &= \lim_{T \to \infty} \Biggl( \int_{Y_T} \vol_Y - \int_{X_T} \vol_X \Biggr) \\
    &\ge \lim_{T \to \infty} \Biggl( \int_{Y_T} \varphi - \int_{X_T} \varphi \Biggr) \\
    &= \lim_{T \to \infty} \int_{Z_T} \varphi = 0
  \end{align*}
  since $\varphi$ is a calibration and $X$ is calibrated. Equality holds if and only if $\varphi \vert_Y = \vol_Y$, that is, if and only if $Y$ is calibrated.
\end{proof}

\section{Index Formula}
\label{sec:index-formula}

In this section, we derive various formulae for the index of the operator of Dirac type that arises as the linearisation of the deformation map under general and special assumptions. We start with the deformation map and its linearisation in \sectionref{subsec:deformation-map}, which is mostly drawn from McLean \cite{McL98}. Then we present a proof of the index formula for closed Cayley submanifolds in \sectionref{subsec:index-formula-closed}, which we then generalise to the asymptotically cylindrical case in \sectionref{subsec:index-formula-acyl}. The formula we get (see \theoremref{thm:main-index-formula}) involves the $\eta$\dash-invariants of two operators. In \sectionref{subsec:relation-Bev-D}, we investigate the relation between these two operators, resulting in an alternative index formula involving the spectral flow (see \propositionref{prop:index-formula-spectral-flow}). We finish this section by deriving simplified formulae under special assumptions on the cross-section at infinity.

\subsection{Deformation Map}
\label{subsec:deformation-map}

Here we present the basic setup, that is, the deformation map and its linearisation. For more details, see \cite{McL98}.

\begin{lemma}
  Let $(M, g)$ be an asymptotically cylindrical manifold with cross-section~$N$, and let $X$ be an asymptotically cylindrical submanifold of~$M$ with cross-section~$Y$. Then there are tubular neighbourhoods $U \subseteq \normal[M]{X}$ and $V \subseteq \normal[N]{Y}$ of the $0$\dash-sections with $(\Psi \vert_{(R, \infty) \times N})^{-1}(U) = (R, \infty) \times V$ for some $R > 0$, where $\Psi \colon (0, \infty) \times N \to M$ is as in \definitionref{def:ACyl}, such that the exponential map (with respect to the metric~$g$) defines an isomorphism from~$U$ to an open neighbourhood of~$X$ in~$M$.
\end{lemma}

So any submanifold that is $\CC^1$\dash-close to~$X$ can be parametrised by a section of the normal bundle~$\normal[M]{X}$ with small $\CC^1$\dash-norm. Furthermore, that submanifold is asymptotically cylindrical if and only if the corresponding section of the normal bundle decays to~$0$ at an exponential rate.

With this preparation, we can define the deformation map as in \cite[Section~6]{McL98},
\begin{equation*}
  F \colon \sections_\lambda(U) \to \sections_\lambda(E) \, \text{,} \quad s \mapsto \pi_E((\exp_s)^\ast(\tau)) \, \text{,}
\end{equation*}
where $U \subseteq \normal[M]{X}$ is an appropriate tubular neighbourhood of the $0$\dash-section. As in \cite{McL98} we get $(\D F)_0(s) = D s$, where $D \colon \sections(\normal[M]{X}) \to \sections(E)$ is defined in the next theorem.

\begin{theorem}[{cf.~\cite[Theorem~6--3]{McL98}}]
  Let $M$ be an asymptotically cylindrical $8$\dash-manifold with an asymptotically cylindrical $\Spin(7)$-structure asymptotic to $(0, \infty) \times N$, where $N$ is a $7$\dash-manifold with a $G_2$\dash-structure, let $X$ be an asymptotically cylindrical Cayley submanifold of~$M$ asymptotic to $(0, \infty) \times Y$, where $Y$ is a closed associative submanifold of~$N$, and let $D \colon \sections(\normal[M]{X}) \to \sections(E)$ be defined by
  \begin{equation}
    D s \defeq \sum_{i = 1}^4 e_i \times \nabla_{e_i}^\perp s + \sum_{i = 5}^8 (\nabla_s \Phi)(e_i, e_2, e_3, e_4) (e_i \times e_1) \, \text{,} \label{eq:def-D}
  \end{equation}
  where $E$ is a vector bundle of rank~$4$ over~$X$ as defined in~\eqref{eq:def-E}, $(e_i)_{i = 1, \dotsc, 4}$ is any positive local orthonormal frame of~$X$, $(e_i)_{i = 5, \dotsc, 8}$ is any local orthonormal frame of~$\normal[M]{X}$, and $\nabla^\perp$ is the induced connection on~$\normal[M]{X}$.
  
  Then the Zariski tangent space to the moduli space of all local deformations of~$X$ as a Cayley submanifold of~$M$ can be identified with the kernel of the operator $D \colon \sections(\normal[M]{X}) \to \sections(E)$,
\end{theorem}

\begin{proposition}
  Let $M$ be an asymptotically cylindrical $8$\dash-manifold with an asymptotically cylindrical $\Spin(7)$-structure asymptotic to $(0, \infty) \times N$, where $N$ is a $7$\dash-manifold with a $G_2$\dash-structure, let $X$ be an asymptotically cylindrical Cayley submanifold of~$M$ asymptotic to $(0, \infty) \times Y$, where $Y$ is a closed associative submanifold of~$N$, and let $0 < \alpha < 1$.
  
  If the operator $D_\lambda \colon \CC^\infty_\lambda(\normal[M]{X}) \to \CC^\infty_\lambda(E)$ defined in~\eqref{eq:def-D} is surjective, then the moduli space of all smooth Cayley submanifolds of~$M$ that are $\CC^{2, \alpha}_\lambda$\dash-close to~$X$ is a smooth manifold of dimension $\dim \ker D_\lambda = \ind_\lambda D$.
\end{proposition}

\subsection{Index Formula for Closed Cayley Submanifolds}
\label{subsec:index-formula-closed}

Here we present a proof of the index formula in the case of closed Cayley submanifolds, which forms the basis for the proof of the more general index formula for asymptotically cylindrical Cayley submanifolds in the \hyperref[subsec:index-formula-acyl]{next section}.

\begin{proposition} \label{prop:index-formula-closed}
  Let $M$ be an $8$\dash-manifold with a $\Spin(7)$-structure, let $X$ be a closed Cayley submanifold of~$M$, and let $D \colon \sections(\normal[M]{X}) \to \sections(E)$ be defined as in~\eqref{eq:def-D}. Then
  \begin{equation}
    \ind D = \frac{1}{2} \chi(X) - \frac{1}{2} \sigma(X) - [X] \cdot [X] \, \text{,} \label{eq:index-formula}
  \end{equation}
  where $\chi(X)$ is the Euler characteristic, $\sigma(X)$ is the signature, and $[X] \cdot [X]$ is the self-intersection number of~$X$.
\end{proposition}

\begin{proof}
  We first recall some facts from \cite[Section~6]{McL98}. Suppose that $X$ has a spin structure, and let $\mathbb{S}^+$ and $\mathbb{S}^-$ denote the positive and negative spinor bundles, respectively (note that both bundles are quaternionic line bundles). Then there is a quaternionic line bundle~$F$ over~$X$ such that
  \begin{equation}
    \normal[M]{X} \otimes_\R \C \cong \mathbb{S}^- \otimes_\C F \quad \text{and} \quad E \otimes_\R \C \cong \mathbb{S}^+ \otimes_\C F \, \text{.} \label{eq:twisted-spin-bundle}
  \end{equation}
  If the $\Spin(7)$-structure is torsion-free, then $D$ can be identified with a negative twisted Dirac operator \cite[Theorem~6--3]{McL98}, that is, the negative Dirac operator associated to the bundle $\mathbb{S} \otimes_\C F$ with the tensor product connection. If the $\Spin(7)$-structure is not torsion-free, then $D$ may not be a negative twisted Dirac operator but the symbol of $D$ is still the same as the symbol of a negative twisted Dirac operator, and the index of an elliptic operator depends only on the symbol of the operator.
  
  The Atiyah--Singer Index Theorem yields
  \begin{equation}
    \ind D = - \int_X (\Ahat(X) \Ch(F))_4 \, \text{,} \label{eq:atiyah-singer}
  \end{equation}
  where $\Ahat(X)$ is the total $\Ahat$-class
  \begin{equation}
    \Ahat(X) = 1 - \frac{1}{24} p_1(X)
  \end{equation}
  and $\Ch(F)$ is the Chern character
  \begin{equation}
    \Ch(F) = 2 + c_1(F) + \frac{1}{2} (c_1(F)^2 - 2 c_2(F)) \, \text{.}
  \end{equation}
  Here $p_1(X)$ is the first Pontryagin class of~$X$ and $c_i(F)$ is the $i$-th Chern class of~$F$. Note that $c_1(F) = 0$ since $F$ has a quaternionic structure. Hence
  \begin{equation}
    - (\Ahat(X) \Ch(F))_4 = \frac{1}{12} p_1(X) + c_2(F) \, \text{.} \label{eq:atiyah-singer-integrand}
  \end{equation}
  Now \eqref{eq:twisted-spin-bundle} implies that
  \begin{equation}
    - p_1(\normal[M]{X}) = c_2(\normal[M]{X} \otimes_\R \C) = 2 c_2(\mathbb{S}^-) + 2 c_2(F) \label{eq:normal-twisted-spin}
  \end{equation}
  since $\Ch(\mathbb{S}^- \otimes_\C F) = \Ch(\mathbb{S}^-) \Ch(F)$ and $c_1(\normal[M]{X} \otimes_\R \C) = c_1(\mathbb{S}^-) = c_1(F) = 0$.
  We further have
  \begin{equation}
    c_2(\mathbb{S}^-) = \frac{1}{2} e(X) - \frac{1}{4} p_1(X) \, \text{,} \label{eq:negative-spinors-c2}
  \end{equation}
  where $e(X)$ is the Euler class of~$X$. Combining \eqref{eq:atiyah-singer}, \eqref{eq:atiyah-singer-integrand}, \eqref{eq:normal-twisted-spin}, and \eqref{eq:negative-spinors-c2} yields
  \begin{equation}
    \ind D = \frac{1}{3} \int_X p_1(X) - \frac{1}{2} \int_X e(X) - \frac{1}{2} \int_X p_1(\normal[M]{X}) \, \text{.}
  \end{equation}
  This formula is valid even when $X$ does not have a spin structure.
  
  The definition~\eqref{eq:def-Phi-0} of the $\Spin(7)$-form~$\Phi_0$ on~$\R^8$ shows that the interior product with the $\Spin(7)$-form $\Phi$ gives an isomorphism
  \begin{equation}
    \altforms^2_- \normal[M]{X} \cong \altforms^2_- T^\ast X \, \text{.}
  \end{equation}
  So
  \begin{equation}
    p_1(\normal[M]{X}) - 2 e(\normal[M]{X}) = p_1(\altforms^2_- \normal[M]{X}) = p_1(\altforms^2_- T^\ast X) = p_1(X) - 2 e(X) \, \text{.}
  \end{equation}
  Hence
  \begin{equation}
    \ind D = \frac{1}{2} \int_X e(X) - \frac{1}{6} \int_X p_1(X) - \int_X e(\normal[M]{X}) \, \text{.}
  \end{equation}
  Now \eqref{eq:index-formula} follows from this formula using the generalised Gauß--Bonnet Theorem
  \begin{equation}
    \int_X e(X) = \chi(X) \, \text{,}
  \end{equation}
  the Hirzebruch Signature Theorem
  \begin{equation}
    \frac{1}{3} \int_X p_1(X) = \sigma(X) \, \text{,}
  \end{equation}
  and the relation
  \begin{equation}
    \int_X e(\normal[M]{X}) = [X] \cdot [X]
  \end{equation}
  between the Euler class of the normal bundle and the self-intersection number.
\end{proof}

\subsection{Index Formula for Asymptotically Cylindrical Cayley Submanifolds}
\label{subsec:index-formula-acyl}

Here we prove a general index formula containing $\eta$\dash-invariants.

\begin{theorem} \label{thm:main-index-formula}
  Let $M$ be an asymptotically cylindrical manifold with an asymptotically cylindrical $\Spin(7)$-structure asymptotic to $(0, \infty) \times N$, where $N$ is a $7$\dash-manifold with torsion-free $G_2$\dash-structure, let $X$ be an asymptotically cylindrical Cayley submanifold asymptotic to $(0, \infty) \times Y$, where $Y$ is a closed associative submanifold of~$N$, and let $D \colon \sections(\normal[M]{X}) \to \sections(E)$ be defined as in~\eqref{eq:def-D}. Then
  \begin{equation}
    \ind_\lambda D = \frac{1}{2} \chi(X) - \frac{1}{2} \sigma(X) - \int_X e(\normal[M]{X}) - \frac{\dim \ker \tilde{D}}{2} + \frac{\eta(\tilde{D}) - \eta(\Bev)}{2} \, \text{.} \label{eq:index-eta}
  \end{equation}
  where $\lambda > 0$ is such that $[- \lambda, 0)$ contains no eigenvalue of~$\tilde{D}$. Here
  \begin{compactenum}[(i)]
    \item $\chi(X)$ is the Euler characteristic of~$X$,
    \item $\sigma(X)$~is the signature of~$X$ (the signature of the non-degenerate quadratic form induced by the cup-product on the image of $H^2_{\text{cs}}(X)$ in $H^2(X)$),
    \item $e(\normal[M]{X})$ is the Euler density,
    \item $\tilde{D} \colon \sections(\normal[N]{Y}) \to \sections(\normal[N]{Y})$ is the (twisted) Dirac operator that arises as the linearisation of the deformation map for associative submanifolds, and
    \item $\Bev \colon \forms^0(Y) \oplus \forms^2(Y) \to \forms^0(Y) \oplus \forms^2(Y)$ is defined by
    \begin{equation*}
      \Bev(f, \alpha) \defeq (\mathord\ast \D \alpha, - \mathord\ast \D f - \D \mathord\ast \alpha) \, \text{.}
    \end{equation*}
  \end{compactenum}
\end{theorem}

\begin{proof}
  Like in the closed case (see the proof of \propositionref{prop:index-formula-closed}), $D$~is a negative twisted Dirac operator. Furthermore, $\tilde{D} \colon \sections(\normal[N]{Y}) \to \sections(\normal[N]{Y})$ satisfies
  \begin{equation*}
    D s = - u \times (- \nabla_u s + \tilde{D} s) \, \text{.}
  \end{equation*}
  So the Atiyah--Patodi--Singer Index Theorem (\theoremref{thm:atiyah-patodi-singer}, \eqref{eq:atiyah-patodi-singer-negative}) implies
  \begin{equation}
    \ind_\lambda D = - \int_X (\Ahat(X) \Ch(F))_4 - \frac{\dim \ker \tilde{D} - \eta(\tilde{D})}{2} \, \text{.}
  \end{equation}
  Like in the closed case (see the proof of \propositionref{prop:index-formula-closed}), we get
  \begin{equation}
    - \int_X (\Ahat(X) \Ch(F))_4 = \frac{1}{2} \int_X e(X) - \frac{1}{6} \int_X p_1(X) - \int_X e(\normal[M]{X}) \, \text{.}
  \end{equation}
  Now \theoremref{thm:atiyah-patodi-singer-signature} implies
  \begin{equation}
    \frac{1}{3} \int_X p_1(X) = \sigma(X) + \eta(\Bev) \, \text{.}
  \end{equation}
  Furthermore, the generalised Gauss--Bonnet--Chern Theorem (\theoremref{thm:gauss-bonnet-chern}, see also \cite[Lemma~9.2]{Mel93}) yields
  \begin{equation}
    \int_X e(X) = \chi(X) \, \text{.}
  \end{equation}
  Now \eqref{eq:index-eta} follows from the above formulae.
\end{proof}

\subsection{Relation Between \texorpdfstring{$\mathbfit{B^{\text{ev}}}$}{B\textasciicircum ev} and \texorpdfstring{$\mathbf{\tilde{\mathbfit{D}}}$}{\textbackslash\textasciitilde D}}
\label{subsec:relation-Bev-D}

Here we investigate the relation between the two operators $\Bev$ and $\tilde{D}$ that appear in the index formula of \theoremref{thm:main-index-formula}. We derive an alternative index formula involving the spectral flow (\propositionref{prop:index-formula-spectral-flow}).

\begin{lemma}
  The operator~$\Bev$ is the Dirac operator associated to the Dirac bundle $\altforms^0 Y \oplus \altforms^2 Y$ with the Clifford multiplication
  \begin{equation*}
    v \cdot (f, \alpha) = (v \interior \hodgestar \alpha, - f \hodgestar v^\flat - v^\flat \wedge \hodgestar \alpha)
  \end{equation*}
  for $v \in T Y$, $f \in \altforms^0 Y$, $\alpha \in \altforms^2Y$ and the Levi-Civita connection.
\end{lemma}

\begin{proof}
  We have
  \begin{align*}
    \sum_{i = 1}^3 e_i \cdot (\nabla_{e_i} f, \nabla_{e_i} \alpha) &= \sum_{i = 1}^3 (e_i \interior \hodgestar \nabla_{e_i} \alpha, - \nabla_{e_i} f \hodgestar e^i - e^i \wedge \hodgestar \nabla_{e_i} \alpha) \\
    &= \Biggl( \hodgestar \sum_{i = 1}^3 e^i \wedge \nabla_{e_i} \alpha, - \hodgestar \sum_{i = 1}^3 e^i \wedge \nabla_{e_i} f + \hodgestar \sum_{i = 1}^3 e_i \interior \nabla_{e_i} \alpha \Biggr) \\
    &= (\hodgestar \D \alpha, - \hodgestar \D f - \hodgestar \updelta \alpha) \\
    &= \Bev(f, \alpha) \, \text{.} \qedhere
  \end{align*}
\end{proof}

\begin{lemma} \label{lem:iso-clifford-bundles}
  Let $s \in \sections(\normal[N]{Y})$ be a non-vanishing section with pointwise norm~$1$ (which always exists since $Y$ is $3$\dash-dimensional and $\normal[N]{Y}$ has rank~$4$), and let
  \begin{equation}
    h \colon \altforms^0 Y \oplus \altforms^2 Y \to \normal[N]{Y} \, \text{,} \quad (f, \alpha) \mapsto f s + s \times (\hodgestar \alpha)^\sharp \, \text{,} \label{eq:def-isomorphism-clifford-bundles}
  \end{equation}
  where ``$\times$'' is the cross product of the $G_2$\dash-structure on~$N$. Then
  \begin{equation*}
    h(v \cdot (f, \alpha)) = v \times h(f, \alpha)
  \end{equation*}
  for $v \in T Y$, $f \in \altforms^0 Y$, and $\alpha \in \altforms^2Y$. So $h$ defines an isomorphism of bundles of Clifford modules.
\end{lemma}

\begin{proof}
  There is a vector-valued $3$\dash-form $\chi \in \forms^3(N, T N)$ such that \cite[(6)]{AS08b}
  \begin{equation}
    \chi(u, v, w) = - u \times (v \times w) - g(u, v) w + g(u, w) v
  \end{equation}
  for all $u, v, w \in \sections(T N)$. Hence
  \begin{equation*}
    v \times (s \times w) = g(v, w) s - s \times (v \times w)
  \end{equation*}
  for all $v, w \in \sections(T Y)$. So
  \begin{align*}
    h(v \cdot (f, \alpha)) &= h((v \interior \hodgestar \alpha, - f \hodgestar v^\flat - v^\flat \wedge \hodgestar \alpha)) \\
    &= (v \interior \hodgestar \alpha) s - s \times (\hodgestar (f \hodgestar v^\flat + v^\flat \wedge \hodgestar \alpha))^\sharp \\
    &= g(v, (\hodgestar \alpha)^\sharp) s + v \times f s- s \times (v \times (\hodgestar \alpha)^\sharp) \\
    &= v \times (f s + s \times (\hodgestar \alpha)^\sharp) \\
    &= v \times h(f, \alpha)
  \end{align*}
  since
  \begin{equation*}
    (\hodgestar (v^\flat \wedge w^\flat))^\sharp = v \times w
  \end{equation*}
  for all $v, w \in \sections(T Y)$ since $Y$ is associative.
\end{proof}

\begin{corollary} \label{cor:relation-D-hat-D-tilde}
  Let $s \in \sections(\normal[N]{Y})$ be a non-vanishing section with pointwise norm~$1$, and let
  \begin{equation}
    \hat{D} \colon \sections(\normal[N]{Y}) \to \sections(\normal[N]{Y}) \, \text{,} \quad z \mapsto h(\Bev(h^{-1}(z))) \, \text{,} \label{eq:def-D-hat}
  \end{equation}
  where $h \colon \altforms^0 Y \oplus \altforms^2 Y \to \normal[N]{Y}$ is defined in \eqref{eq:def-isomorphism-clifford-bundles}. Then the operators $\hat{D}$ and $\tilde{D}$ have the same symbol. Moreover, if $s$ is parallel, then $\hat{D} = \tilde{D}$.
\end{corollary}

\begin{proposition} \label{prop:index-formula-spectral-flow}
  Let $s \in \sections(\normal[N]{Y})$ be a non-vanishing section with pointwise norm~$1$, let $\hat{D} \colon \sections(\normal[N]{Y}) \to \sections(\normal[N]{Y})$ be defined as in \eqref{eq:def-D-hat}, and let $D_t \colon \sections(\normal[N]{Y}) \to \sections(\normal[N]{Y})$ ($0 \le t \le 1$) be a smooth family of linear differential operators with the same symbol such that $D_0 = \hat{D}$ and $D_1 = \tilde{D}$. Then
  \begin{equation}
    \ind_\lambda D = \frac{1}{2} \chi(X) - \frac{1}{2} \sigma(X) - e(\normal[M]{X}, s) + \SF(D_t) - \frac{\dim \ker \tilde{D}}{2} \, \text{,} \label{eq:index-spectral-flow}
  \end{equation}
  where $\SF(D_t)$ is the spectral flow of the family $(D_t)$, that is, the number of eigenvalues that go from $< 0$ to $\ge 0$ minus the number of eigenvalues that go from $\ge 0$ to $< 0$.
\end{proposition}

\begin{proof}
  First note that the spectral flow does not depend on the choice of the smooth family \cite[Theorem~(7$\cdot$4)]{APS76}. So let $\mu \colon \R \to [0, 1]$ be a smooth function such that $\mu(t) = 0$ for all $t \le -1$ and $\mu(t) = 1$ for all $t \ge 1$, and let $P \colon \sections(\R \times \normal[N]{Y}) \to \sections(\R \times \normal[N]{Y})$ (over the manifold $\R \times Y$) be defined by
  \begin{equation}
    P s \defeq - \nabla_t s + \mu(t) \tilde{D} s + (1 - \mu(t)) \hat{D} s \, \text{.}
  \end{equation}
  Then $\ind_\lambda P + \frac{1}{2} (\dim \ker \tilde{D} + \dim \ker \hat{D})$ (for $\lambda > 0$ such that $[- \lambda, 0)$ contains no eigenvalue of~$\Bev$ or~$\tilde{D}$) is equal to the spectral flow from $\hat{D}$ to $\tilde{D}$ \cite[Section~7]{APS76}.
  
  Note that $P$ is the negative Dirac operator with respect to the connection $\mu(t) \nabla + (1 - \mu(t)) \hat{\nabla}$, where $\hat{\nabla} \defeq (h^{-1})^\ast \nabla$. So the Atiyah--Patodi--Singer Index Theorem (\theoremref{thm:atiyah-patodi-singer}, \eqref{eq:atiyah-patodi-singer-negative}) implies
  \begin{align}
    &\ind_\lambda P \notag \\
    &= - \int_{\R \times Y} \Ahat(\R \times Y) \Ch(F \vert_Y) - \frac{\dim \ker \tilde{D} - \eta(\tilde{D})}{2} - \frac{\dim \ker \hat{D} + \eta(\hat{D})}{2}
  \end{align}
  Note that
  \begin{equation}
    \dim \ker \hat{D} = \dim \ker \Bev = b^0(Y) + b^1(Y) \quad \text{and} \quad \eta(\hat{D}) = \eta(\Bev) \, \text{.}
  \end{equation}
  As before, we get
  \begin{align}
    &- \int_{\R \times Y} (\Ahat(\R \times Y) \Ch(F \vert_Y))_4 \notag \\
    &= \frac{1}{2} \int_{\R \times Y} e(\R \times Y) - \frac{1}{6} \int_{\R \times Y} p_1(\R \times Y) - \int_{\R \times Y} e(\R \times \normal[N]{Y}) \, \text{.}
  \end{align}
  Here $e(\R \times Y) = p_1(\R \times Y) = 0$. The generalised Gauss--Bonnet--Chern Theorem (\theoremref{thm:gauss-bonnet-chern}) implies
  \begin{align}
    \int_X e(\normal[M]{X}) &= e(\normal[M]{X}, s) - \int_Y \Theta(F_\nabla \vert_Y, \nabla s) \notag \\
    &= e(\normal[M]{X}, s) + \int_{\R \times Y} e(\R \times \normal[N]{Y}) \, \text{.} \label{eq:application-gauss-bonnet-chern-normal-bundle}
  \end{align}
  Now \eqref{eq:index-spectral-flow} follows from the above formulae.
\end{proof}

\begin{remark}
  Note that both $e(\normal[M]{X}, s)$ and $\SF(D_t)$ depend on the choice of~$s$ but their difference does not. In particular, if $s$ and $s^\prime$ are two non-vanishing sections of $\normal[N]{Y}$ (\wolog\ with pointwise norm~$1$) that are homotopic through non-vanishing sections, then $e(\normal[M]{X}, s) = e(\normal[M]{X}, s^\prime)$. Hence also $\SF(D_t) = \SF(D_t^\prime)$.
\end{remark}

\subsection{Additional Assumptions}
\label{subsec:additional-assumptions}

Here we prove simplified versions of the index formula under special assumptions on the cross-section at infinity. Examples where these assumptions are satisfied can usually be found if the cross-section at infinity of the $\Spin(7)$-manifold has reduced holonomy. In particular, we will apply the results of this section to asymptotically cylindrical Cayley submanifolds inside the asymptotically cylindrical Riemannian $8$\dash-manifolds with holonomy $\Spin(7)$ constructed by Kovalev in \cite{Kov13} in \sectionref{sec:examples}.

\begin{proposition} \label{prop:parallel-section}
  Let $M$ be an asymptotically cylindrical manifold with an asymptotically cylindrical $\Spin(7)$-structure asymptotic to $(0, \infty) \times N$, where $N$ is a $7$\dash-manifold with torsion-free $G_2$\dash-structure, let $X$ be an asymptotically cylindrical Cayley submanifold asymptotic to $(0, \infty) \times Y$, where $Y$ is a closed associative submanifold of~$N$, and let $D \colon \sections(\normal[M]{X}) \to \sections(E)$ be defined as in~\eqref{eq:def-D}. Suppose that the normal bundle $\normal[N]{Y}$ of~$Y$ in~$N$ has a non-trivial parallel section~$s \in \sections(\normal[N]{Y})$. Then
  \begin{equation}
    \ind_\lambda D = \frac{1}{2} \chi(X) - \frac{1}{2} \sigma(X) - e(\normal[M]{X}, s) - \frac{b^0(Y) + b^1(Y)}{2} \, \text{,} \label{eq:index-spectral-flow-parallel-section}
  \end{equation}
  where $\lambda > 0$ is such that $[- \lambda, 0)$ contains no eigenvalue of~$\Bev$.
\end{proposition}

\begin{proof}
  We have $\hat{D} = \tilde{D}$ by \corollaryref{cor:relation-D-hat-D-tilde} since $s$ is parallel. Hence $\eta(\tilde{D}) = \eta(\Bev)$ and $\dim \ker \tilde{D} = \dim \ker \Bev$. So \eqref{eq:index-spectral-flow-parallel-section} follows from \eqref{eq:index-eta} using the generalised Gauss--Bonnet--Chern Theorem (\theoremref{thm:gauss-bonnet-chern}).
\end{proof}

\begin{proposition} \label{prop:parallel-section-lift}
  Let $M$ be an asymptotically cylindrical manifold with an asymptotically cylindrical $\Spin(7)$-structure asymptotic to $(0, \infty) \times N$, where $N$ is a $7$\dash-manifold with torsion-free $G_2$\dash-structure, let $X$ be an asymptotically cylindrical Cayley submanifold asymptotic to $(0, \infty) \times Y$, where $Y$ is a closed associative submanifold of~$N$, let $D \colon \sections(\normal[M]{X}) \to \sections(E)$ be defined as in~\eqref{eq:def-D}, and let $\widebar{X}$ be the compactification of~$X$ so that $\partial \widebar{X} = Y$. Suppose that there are a $4$\dash-dimensional manifold~$\widetilde{X}$ and a free involution $\rho \colon \widetilde{X} \to \widetilde{X}$ such that $\widebar{X} \cong \widetilde{X} / \inner{\rho}$. Let $\pi \colon \widetilde{X} \to \widetilde{X} / \inner{\rho} \cong \widebar{X}$ be the projection, and let $\widetilde{Y} \defeq \pi^{-1}(Y)$ (so that $\partial \widetilde{X} = \widetilde{Y}$). Suppose that $(\pi \vert_{\widetilde{Y}})^\ast(\normal[N]{Y})$ has a non-trivial parallel section $s \in \sections((\pi \vert_{\widetilde{Y}})^\ast(\normal[N]{Y}))$ such that
  \begin{equation*}
    s \circ \rho = - s \, \text{.}
  \end{equation*}
  Then
  \begin{equation}
    \begin{split}
      \ind_\lambda D &= \frac{1}{2} \chi(X) + \frac{1}{2} \sigma(X) - \frac{1}{2} \sigma(\widetilde{X}) - \frac{1}{2} e(\pi^\ast(\normal[M]{X}), s) \\
      &\phantom{{}={}} {}+ \frac{b^0(Y) + b^1(Y)}{2} - \frac{b^0(\widetilde{Y}) + b^1(\widetilde{Y})}{2} \, \text{,}
    \end{split} \label{eq:index-parallel-section-lift}
  \end{equation}
  where $\lambda > 0$ is such that $[- \lambda, 0)$ contains no eigenvalue of~$\tilde{D}$.
\end{proposition}

\begin{proof}
  The section~$s$ defines a flat subbundle~$K$ of~$\normal[N]{Y}$ of rank~$1$. Now \lemmaref{lem:iso-clifford-bundles} defines an isomorphism $\normal[N]{X} \cong (\altforms^0 Y \oplus \altforms^2 Y) \otimes K$. Furthermore, the connections are identified via this isomorphism since $s$ is parallel and the $G_2$\dash-structure is torsion-free. In particular, under this isomorphism, $\tilde{D}$~is identified with the operator
  \begin{align*}
    \forms^0(Y, K) \oplus \forms^2(Y, K) &\to \forms^0(Y, K) \to \forms^2(Y, K) \, \text{,} \\
    (f, \alpha) &\mapsto (\mathord\ast \D \alpha, - \mathord\ast \D f - \D \mathord\ast \alpha) \, \text{.}
  \end{align*}
  Hence
  \begin{equation}
    \eta(\tilde{D}) - \eta(\Bev) = \rho_\xi(Y) \, \text{,}
  \end{equation}
  where $\xi \colon \pi(Y) \to \set{\pm 1}$ is the holonomy representation of~$K$ and $\rho_\xi(Y)$ is defined in \cite[Theorem~(2$\cdot$4)]{APS75b}. Now \cite[Theorem~(2$\cdot$4)]{APS75b} and \cite[Lemma~(2$\cdot$5)]{APS75b} imply
  \begin{equation}
    \rho_\xi(Y) = 2 \sigma(X) - \sigma(\widetilde{X}) \, \text{.}
  \end{equation}
  Also note that
  \begin{equation}
    \dim \ker \tilde{D} = b^0(\widetilde{Y})^{- \rho} + b^1(\widetilde{Y})^{- \rho} = (b^0(\widetilde{Y}) - b^0(Y)) + (b^1(\widetilde{Y}) - b^1(Y)) \, \text{.}
  \end{equation}
  Furthermore,
  \begin{equation}
    \int_X e(\normal[M]{X}) = \frac{1}{2} \int_{\widetilde{X}} e(\pi^\ast(\normal[M]{X})) = \frac{1}{2} e(\pi^\ast(\normal[M]{X}), s)
  \end{equation}
  by the generalised Gauss--Bonnet--Chern Theorem (\theoremref{thm:gauss-bonnet-chern}) since $\pi \colon \widetilde{X} \to \widebar{X}$ is a $2$\dash-fold cover and $s$ is a non-trivial parallel section of $\pi^\ast(\normal[M]{X}) \vert_{\partial \widetilde{X}}$. Now \eqref{eq:index-parallel-section-lift} follows from \eqref{eq:index-eta} using the above formulae.
\end{proof}

\begin{proposition} \label{prop:complex}
  Let $M$ be an asymptotically cylindrical manifold with an asymptotically cylindrical $\Spin(7)$-structure asymptotic to $(0, \infty) \times N$, where $N$ is a $7$\dash-manifold with torsion-free $G_2$\dash-structure, let $X$ be an asymptotically cylindrical Cayley submanifold asymptotic to $(0, \infty) \times Y$, where $Y$ is a closed associative submanifold of~$N$, and let $D \colon \sections(\normal[M]{X}) \to \sections(E)$ be defined as in~\eqref{eq:def-D}. Suppose that $N = S^1 \times C$, where $C$ is a Calabi--Yau $3$\dash-fold, and that $Y = S^1 \times Z$, where $Z$ is a complex curve in~$C$. Furthermore, let $s \in \sections(\normal[N]{Y})$ be a non-vanishing section that is invariant under rotations of~$S^1$. Then
  \begin{equation}
    \ind_\lambda D = \frac{1}{2} \chi(X) - \frac{1}{2} \sigma(X) - e(\normal[M]{X}, s) - \frac{\dim_\C H^0(Z, \normal[C]{Z})}{2} \, \text{.} \label{eq:index-complex}
  \end{equation}
  where $\lambda > 0$ is such that $[- \lambda, 0)$ contains no eigenvalue of~$\tilde{D}$.
\end{proposition}

\begin{remark} \label{rmk:complex-compactification}
  Under the above hypotheses, consider the map
  \begin{equation*}
    (0, \infty) \times S^1 \times C \to \C \times C \, \text{,} \quad (t, \E^{\I \theta}, x) \mapsto (\E^{- t + \I \theta}, x) \, \text{.}
  \end{equation*}
  This is a diffeomorphism onto $\set{z \in \C \colon 0 < \abs{z} < 1} \times C$. Let $\widebar{M}$ be the compactification of~$M$ that is obtained by using this diffeomorphism and extending it to $\set{0} \times C$ (so $\widebar{M}$ is a closed manifold with $\widebar{M} \setminus M \cong C$). Extend $X$ similarly to~$\widebar{X}$. Then we can identify $\normal[\widebar{M}]{\widebar{X}}$ on the cylindrical end with $\set{z \in \C \colon \abs{z} < 1} \times \normal[\widebar{M}]{C}$ (where we identify $C \cong \widebar{M} \setminus M$). Now the hypothesis that $s$ is invariant under rotations of~$S^1$ means that $s$ induces a non-vanishing section of $\normal[\widebar{M}]{C}$. Hence $e(\normal[M]{X}, s) = e(\normal[\widebar{M}]{\widebar{X}})$, which is equal to the self-intersection number of~$\widebar{X}$ in~$\widebar{M}$.
\end{remark}

\begin{proof}
  First note that
  \begin{equation}
    \eta(\Bev) = 0
  \end{equation}
  by \lemmaref{lem:bev-zero}. Let $J \colon \normal[N]{Y} \to \normal[N]{Y}$ be the parallel, orthogonal almost complex structure on~$\normal[N]{Y}$ induced by the complex structure on~$\normal[C]{Z}$. Then $J(s) = v \times s$ for $v \defeq - \frac{\partial}{\partial \theta} \in \sections(T Y)$, where $\theta$ is the coordinate on the $S^1$\dash-factor. Furthermore,
  \begin{equation*}
    w \times J s + J(w \times s) = w \times (v \times s) + v \times (w \times s) = - 2 g(v, w) s
  \end{equation*}
  for all $w \in \sections(T Y)$, $s \in \sections(\normal[N]{Y})$ by \cite[(6)]{AS08b}. So
  \begin{equation*}
    \tilde{D}(J s) + J(\tilde{D} s) = - 2 \nabla_v s \, \text{.}
  \end{equation*}
  Let $\rho \colon Y \to Y$, $\rho(\E^{\I \theta}, x) \defeq (\E^{- \I \theta}, x)$ for $(\E^{\I \theta}, x) \in S^1 \times Z \cong Y$. Then $v \circ \rho = - v$. So $\nabla_v (s \circ \rho) = - (\nabla_v s) \circ \rho$. Hence
  \begin{equation*}
    \tilde{D}(s \circ \rho) - (\tilde{D} s) \circ \rho = - 2 (J \nabla_v s) \circ \rho \, \text{.}
  \end{equation*}
  Thus, if
  \begin{equation*}
    \tilde{J} \colon \sections(\normal[N]{Y}) \to \sections(\normal[N]{Y}) \, \text{,} \quad s \mapsto (J s) \circ \rho \, \text{,}
  \end{equation*}
  then
  \begin{equation*}
    \tilde{D}(\tilde{J} s) = \tilde{D}((J s) \circ \rho) = (\tilde{D}(J s)) \circ \rho + 2 (\nabla_v s) \circ \rho = - (J(\tilde{D} s)) \circ \rho = - \tilde{J}(\tilde{D} s) \, \text{.}
  \end{equation*}
  This implies that the spectrum of~$\tilde{D}$ is symmetric. Hence also
  \begin{equation}
    \eta(\tilde{D}) = 0 \, \text{.}
  \end{equation}
  Note that $\ker \tilde{D}$ is isomorphic to the space of holomorphic sections of the normal bundle~$\normal[C]{Z}$ \cite[Lemma~5.11]{CHNP12}. Hence
  \begin{equation}
    \dim \ker \tilde{D} = \dim_\C H^0(Z, \normal[C]{Z}) \, \text{.}
  \end{equation}
  Furthermore, $\nabla_v s = 0$ since $s$~is invariant under rotations of~$S^1$. So the generalised Gauss--Bonnet--Chern Theorem (\theoremref{thm:gauss-bonnet-chern}) implies
  \begin{equation}
    \int_X e(\normal[M]{X}) = e(\normal[M]{X}, s)
  \end{equation}
  since $N = S^1 \times C$ is endowed with the product metric. Now \eqref{eq:index-complex} follows from \eqref{eq:index-eta} using the above formulae.
\end{proof}

\section{Varying the Spin(7)-Structure}
\label{sec:varying-spin7}

In this section, we prove that under certain genericity assumptions on the $\Spin(7)$-structure, asymptotically cylindrical Cayley submanifolds form a smooth finite-dimensional moduli space. There are various versions, depending on the precise conditions on the asymptotically cylindrical $\Spin(7)$-structures allowed and whether the cross-section at infinity of the Cayley submanifold is fixed.

\begin{definition}
  Let $X$ be a topological space. We say that a statement holds for \emph{generic} $x \in X$ if the set of all $x \in X$ for which the statement is true is a residual set, that is, it contains a set which is the intersection of countably many open dense subsets.
\end{definition}

In the following theorems we use the $\CC^\infty_\lambda$\dash-topology for the space of all $\Spin(7)$-structures.

\begin{theorem}
  Let $M$ be an asymptotically cylindrical $8$\dash-manifold with an asymptotically cylindrical $\Spin(7)$-structure asymptotic to $(0, \infty) \times N$, where $N$ is a $7$\dash-manifold with a $G_2$\dash-structure, let $X$ be an asymptotically cylindrical Cayley submanifold of~$M$ asymptotic to $(0, \infty) \times Y$, where $Y$ is a closed associative submanifold of~$N$, and let $0 < \alpha < 1$.
  
  If $\lambda > 0$ is small enough, then for every generic asymptotically cylindrical $\Spin(7)$-structure~$\Psi$ with rate~$\lambda$ that is $\CC^{2, \alpha}_\lambda$-close to~$\Phi$, the moduli space of all asymptotically cylindrical Cayley submanifolds of $(M, \Psi)$ with rate~$\lambda$ that are $\CC^{2, \alpha}_\lambda$-close to~$X$ is either empty or a smooth manifold of dimension $\ind_\lambda D$, where $D$ is defined in~\eqref{eq:def-D}.
\end{theorem}

\begin{proof}
  The proof works mostly analogous to the proof of \cite[Theorem~3.7]{Ohs14}. So we will only explain the main differences.
  
  Let $U \subseteq \normal[M]{X}$, $V \subseteq \altforms^4_1 M \oplus \altforms^4_7 M \oplus \altforms^4_{35} M$, and $\Theta \colon V \to \altforms^4 M$ be as in the proof of \cite[Theorem~3.7]{Ohs14}, and let
  \begin{equation*}
    \tilde{F} \colon \sections_\lambda(U) \oplus \sections_\lambda(V) \to \sections_\lambda(E) \, \text{,} \quad (s, \chi) \mapsto \pi_E((\exp_s)^\ast(\tau_{\Theta(\chi)})) \, \text{.}
  \end{equation*}
  Then
  \begin{equation*}
    (\D \tilde{F})_{(0, 0)}(s, 0) = D s
  \end{equation*}
  for $s \in \sections_\lambda(\normal[M]{X})$ and \cite[Lemma~3.9]{Ohs14}
  \begin{equation*}
    (\D \tilde{F})_{(0, 0)}(0, \chi) = - \pi_E(e \vert_X) \, \text{,}
  \end{equation*}
  where $e \in \sections_\lambda(\altforms^2_7 M)$ and $\chi \defeq g(\tau_\Phi, e) \in \sections_\lambda(\altforms^4_7 M)$. Now $\tilde{F}$ extends to a map
  \begin{equation*}
    \tilde{F}^{2, \alpha}_\lambda \colon \CC^{2, \alpha}_\lambda(U) \oplus \CC^{2, \alpha}_\lambda(V) \to \CC^{1, \alpha}_\lambda(E)
  \end{equation*}
  of class~$\CC^1$. Since the image of the operator
  \begin{equation*}
    D^{2, \alpha}_\lambda \colon \CC^{2, \alpha}_\lambda(\normal[M]{X}) \to \CC^{1, \alpha}_\lambda(E)
  \end{equation*}
  is closed and has finite codimension and $\CC^\infty_\lambda(E)$ is dense in $\CC^{1, \alpha}_\lambda(E)$, the map
  \begin{equation*}
    (\D \tilde{F}^{2, \alpha}_\lambda)_{(0, 0)} \colon \CC^{2, \alpha}_\lambda(\normal[M]{X}) \oplus \CC^{2, \alpha}_\lambda(\altforms^4_1 M \oplus \altforms^4_7 M \oplus \altforms^4_{35} M) \to \CC^{1, \alpha}_\lambda(E)
  \end{equation*}
  is surjective by \lemmaref{lemma:sum-dense} below.
  
  \begin{lemma}
    Suppose that $s \in \CC^{2, \alpha}_\lambda(U)$ and $\chi \in \CC^{k, \alpha}_\lambda(V)$ (for some $k \ge 2$) satisfy $\tilde{F}^{2, \alpha}_\lambda(s, \chi) = 0$. Then $s \in \CC^{k + 1, \alpha}_\lambda(\normal[M]{X})$.
  \end{lemma}
  
  \begin{proof}
    This follows by considering the quasilinear elliptic equation
    \begin{equation*}
      D^\ast \tilde{F}(s, \chi) = 0
    \end{equation*}
    end elliptic bootstrapping.
  \end{proof}
  
  Note further that if $s \in \CC^{2, \alpha}_\lambda(U)$ and $\chi \in \CC^{2, \alpha}_\lambda(V)$ have small $\CC^{2, \alpha}_\lambda$-norm, then
  \begin{equation*}
    (\D \tilde{F}^{2, \alpha}_\lambda)_{(s, \chi)} \vert_{\CC^{2, \alpha}_\lambda(\normal[M]{X})} \colon \CC^{2, \alpha}_\lambda(\normal[M]{X}) \to \CC^{1, \alpha}_\lambda(E)
  \end{equation*}
  and $D^{2, \alpha}_\lambda$ are asymptotic to the same operator since $\lambda > 0$. In particular, they are Fredholm for the same values of~$\lambda$.
  
  The remaining arguments work like in the closed case \cite[Theorem~3.7]{Ohs14}.
\end{proof}

\begin{lemma} \label{lemma:sum-dense}
  Let $X$ be a topological vector space, let $V$ be closed subspace of~$X$ with finite codimension, and let $W$ be a dense subspace of~$X$. Then $X = V + W$.
\end{lemma}

\begin{proof}
  The proof works by induction. It is trivially true if $V = X$. Otherwise, $X \setminus V$ is open and non-empty. Hence there is some $w \in W \setminus V$ since $W$ is dense. Then $V + \inner{w}$ is a closed subspace of~$X$ whose codimension is smaller than the codimension of~$V$.
\end{proof}

\begin{theorem}
  Let $M$ be an asymptotically cylindrical $8$\dash-manifold with an asymptotically cylindrical $\Spin(7)$-structure asymptotic to $(0, \infty) \times N$, where $N$ is a $7$\dash-manifold with a $G_2$\dash-structure, let $X$ be an asymptotically cylindrical Cayley submanifold of~$M$ asymptotic to $(0, \infty) \times Y$, where $Y$ is a closed associative submanifold of~$N$, and let $0 < \alpha < 1$. Suppose that the moduli space of all associative submanifolds of~$N$ that are $\CC^{2, \alpha}$\dash-close to~$Y$ contains a smooth manifold~$\mathcal{Y}$ with $Y \in \mathcal{Y}$.
  
  If $\lambda > 0$ is small enough, then for every generic asymptotically cylindrical $\Spin(7)$-structure~$\Psi$ with rate~$\lambda$ that is $\CC^{2, \alpha}_\lambda$-close to~$\Phi$, the moduli space of all asymptotically cylindrical Cayley submanifolds of $(M, \Psi)$ that are $\CC^{2, \alpha}$-close to~$X$ and asymptotic to $(0, \infty) \times Y^\prime$ with rate~$\lambda$ for some $Y^\prime \in \mathcal{Y}$ is either empty or a smooth manifold of dimension $\ind_\lambda D + \dim \mathcal{Y}$, where $D$ is defined in~\eqref{eq:def-D}.
\end{theorem}

\begin{proof}
  Let $\rho \colon \R \to [0, 1]$ be a smooth function with $\rho(t) = 0$ for $t \le 0$, $\rho(t) = 1$ for large $t$, and $\norm{\rho^\prime}_{\CC^{1, \alpha}}$ sufficiently small. For $\hat{s} \in \sections(\normal[N]{Y})$, view it as a translation-invariant section of $\R \times \normal[N]{Y}$ (i.e.,~$\nabla_t \hat{s} = 0$), and consider $\rho(t) \hat{s} \in \sections(\normal[M]{X})$. Now consider $\mathcal{Y}$ as a submanifold of $\CC^\infty(\normal[N]{Y})$, and let
  \begin{equation*}
    \tilde{F} \colon \sections_\lambda(U) \oplus \mathcal{Y} \oplus \sections_\lambda(V) \to \sections_\lambda(E) \, \text{,} \quad (s, \hat{s}, \chi) \mapsto \pi_E((\exp_{s + \rho(t) \hat{s}})^\ast(\tau_{\Theta(\chi)})) \, \text{.}
  \end{equation*}
  Note that the image is indeed in $\sections_\lambda(E)$ since $\mathcal{Y}$ consists of associative submanifolds (which implies that $(\exp_{\hat{s}})^\ast(\tau_\Phi) = 0$ on $\R \times Y$ for $\hat{s} \in \mathcal{Y}$). We have
  \begin{equation*}
    (\D \tilde{F})_{(0, 0, 0)}(s, \hat{s}, 0) = D s + \rho^\prime(t) \hat{s}
  \end{equation*}
  for $s \in \sections_\lambda(\normal[M]{X})$, $\hat{s} \in T_0 \mathcal{Y}$. Note that this is a Fredholm operator with index $\ind_\lambda D + \dim \mathcal{Y}$ since $\norm{\rho^\prime}_{\CC^{1, \alpha}}$ is sufficiently small. We further have
  \begin{equation*}
    (\D \tilde{F})_{(0, 0, 0)}(0, 0, \chi) = - \pi_E(e \vert_X) \, \text{,}
  \end{equation*}
  where $e \in \sections_\lambda(\altforms^2_7 M)$ and $\chi \defeq g(\tau_\Phi, e) \in \sections_\lambda(\altforms^4_7 M)$. Now $\tilde{F}$ extends to a map
  \begin{equation*}
    \tilde{F}^{2, \alpha}_\lambda \colon \CC^{2, \alpha}_\lambda(U) \oplus \mathcal{Y} \oplus \CC^{2, \alpha}_\lambda(V) \to \CC^{1, \alpha}_\lambda(E)
  \end{equation*}
  of class~$\CC^1$. Since the image of the operator
  \begin{equation*}
    \CC^{2, \alpha}_\lambda(\normal[M]{X}) \oplus T_0 \mathcal{Y} \to \CC^{1, \alpha}_\lambda(E) \, \text{,} \quad (s, \hat{s}) \mapsto D s + \rho^\prime(t) \hat{s}
  \end{equation*}
  is closed and has finite codimension and $\CC^\infty_\lambda(E)$ is dense in $\CC^{1, \alpha}_\lambda(E)$, the map
  \begin{equation*}
    (\D \tilde{F}^{2, \alpha}_\lambda)_{(0, 0, 0)} \colon \CC^{2, \alpha}_\lambda(\normal[M]{X}) \oplus T_0 \mathcal{Y} \oplus \CC^{2, \alpha}_\lambda(\altforms^4_1 M \oplus \altforms^4_7 M \oplus \altforms^4_{35} M) \to \CC^{1, \alpha}_\lambda(E)
  \end{equation*}
  is surjective.
\end{proof}

\begin{theorem}
  Let $M$ be an asymptotically cylindrical $8$\dash-manifold with an asymptotically cylindrical $\Spin(7)$-structure asymptotic to $(0, \infty) \times N$, where $N$ is a $7$\dash-manifold with a $G_2$\dash-structure, let $X$ be an asymptotically cylindrical Cayley submanifold of~$M$ asymptotic to $(0, \infty) \times Y$, where $Y$ is a closed associative submanifold of~$N$, and let $0 < \alpha < 1$.
  
  If $\lambda > 0$ is small enough, then for every generic asymptotically cylindrical $\Spin(7)$-structure~$\Psi$ with rate~$\lambda$ that is $\CC^{2, \alpha}$-close to~$\Phi$, the moduli space of all asymptotically cylindrical Cayley submanifolds of $(M, \Psi)$ with rate~$\lambda$ that are $\CC^{2, \alpha}$-close to~$X$ is either empty or a smooth manifold of dimension $\ind_\lambda D$, where $D$ is defined in~\eqref{eq:def-D}.
\end{theorem}

\begin{proof}
  Here we have the deformation map
  \begin{equation*}
    \begin{split}
      \tilde{F} \colon \sections_\lambda(U) \oplus \sections(\hat{U}) \oplus \sections_\lambda(V) \oplus \sections(\hat{V}) &\to \sections_\lambda(E) \oplus \sections(\normal[N]{Y}) \, \text{,} \\
      (s, \hat{s}, \chi, \phi) &\mapsto \pi_E((\exp_{s + \rho(t) \hat{s}})^\ast(\tau_{\Theta(\chi + \rho(t) \phi)})) \, \text{,}
    \end{split}
  \end{equation*}
  where $\hat{U} \subseteq \normal[N]{Y}$ and $\hat{V} \subseteq \altforms^3 N$ are appropriate open tubular neighbourhoods of the $0$\dash-sections and the image of this map lives in the image of the map $\sections_\lambda(E) \oplus \sections(\normal[N]{Y}) \to \sections(E)$, $(e, v) \mapsto e + \rho(t) v$ (recall that $E \vert_Y \cong \normal[N]{Y}$). We have
  \begin{equation*}
    (\D \tilde{F})_{(0, 0, 0, 0)}(s, \hat{s}, 0, 0) = (D s + \rho^\prime(t) \hat{s}, \tilde{D} \hat{s})
  \end{equation*}
  for $s \in \sections_\lambda(\normal[M]{X})$, $\hat{s} \in \sections(\normal[N]{Y})$. Note that this is a Fredholm operator with index $\ind_\lambda D$ since $\ind \tilde{D} = 0$ and $\norm{\rho^\prime}_{\CC^{1, \alpha}}$ is sufficiently small. We further have
  \begin{equation*}
    (\D \tilde{F})_{(0, 0, 0, 0)}(0, 0, \chi, \phi) = (- \pi_E(e \vert_X), \pi_\nu(v \vert_Y)) \, \text{,}
  \end{equation*}
  where $e \in \sections_\lambda(\altforms^2_7 M)$, $\chi \defeq g(\tau_\Phi, e) \in \sections_\lambda(\altforms^4_7 M)$, $v \in \sections(T N)$, and $\phi \defeq \mathord\ast_N(\varphi \wedge v^\flat) \in \sections(\altforms^3_7 N)$. Now $\tilde{F}$ extends to a map
  \begin{equation*}
    \tilde{F}^{2, \alpha}_\lambda \colon \CC^{2, \alpha}_\lambda(U) \oplus \CC^{2, \alpha}(\hat{U}) \oplus \CC^{2, \alpha}_\lambda(V) \oplus \CC^{2, \alpha}(\hat{V}) \to \CC^{1, \alpha}_\lambda(E) \oplus \CC^{1, \alpha}(\normal[N]{Y})
  \end{equation*}
  of class~$\CC^1$. Since the image of the operator
  \begin{equation*}
    \begin{split}
      \CC^{2, \alpha}_\lambda(\normal[M]{X}) \oplus \CC^{2, \alpha}(\normal[N]{Y}) &\to \CC^{1, \alpha}_\lambda(E) \oplus \CC^{1, \alpha}(\normal[N]{Y}) \, \text{,} \\
      (s, \hat{s}) &\mapsto (D s + \rho^\prime(t) \hat{s}, \tilde{D} \hat{s})
    \end{split}
  \end{equation*}
  is closed and has finite codimension and $\CC^\infty_\lambda(E)$ and $\CC^\infty(\normal[N]{Y})$ are dense in $\CC^{1, \alpha}_\lambda(E)$ and $\CC^{1, \alpha}(\normal[N]{Y})$, respectively, the map
  \begin{equation*}
    \begin{split}
      (\D \tilde{F}^{2, \alpha}_\lambda)_{(0, 0, 0, 0)} \colon &\CC^{2, \alpha}_\lambda(\normal[M]{X}) \oplus \CC^{2, \alpha}(\normal[N]{Y}) \\
      &\oplus \CC^{2, \alpha}_\lambda(\altforms^4_1 M \oplus \altforms^4_7 M \oplus \altforms^4_{35} M) \oplus \CC^{2, \alpha}(\altforms^3 N) \\
      &\to \CC^{1, \alpha}_\lambda(E) \oplus \CC^{1, \alpha}(\normal[N]{Y})
    \end{split}
  \end{equation*}
  is surjective.
\end{proof}

\section{Examples}
\label{sec:examples}

In this section, we provide examples of asymptotically cylindrical Cayley submanifolds inside the asymptotically cylindrical Riemannian $8$\dash-manifolds with holonomy $\Spin(7)$ constructed by Kovalev in \cite{Kov13} and calculate the indices of these Cayley submanifolds.

\paragraph{Notation.} Throughout this section, if $M$ is a manifold and $n$ is a positive integer, we write $n \, M$ for the connected sum of $n$~copies of~$M$ (the connected sum does not depend on the embeddings of the discs along which the manifolds are glued; furthermore, the connected sum is commutative and associative up to orientation preserving diffeomorphism). Note that $n \, M$ is cobordant to $\coprod_{i = 1}^n M$, the disjoint union of $n$~copies of~$M$. In particular, if $M$ is null-cobordant, so is $n \, M$, and a null-cobordism of~$M$ determines a null-cobordism of $n \, M$ by composing with the cobordism between $n \, M$ and $\coprod_{i = 1}^n M$.

\subsection{Cayley Submanifold in a Spin(7)-Manifold Constructed from \texorpdfstring{$\mathbf{CP^4_{1, 1, 1, 1, 4}}$}{CP\textasciicircum 4\_\{1,1,1,1,4\}}}
\label{subsec:example-1}

We will now construct an asymptotically cylindrical Cayley submanifold inside the asymptotically cylindrical $\Spin(7)$-manifold constructed in \cite[Section~6.2]{Kov13} and compute the index. Let $V \defeq \CP^4_{1, 1, 1, 1, 4}$ (a weighted projective space). This is an orbifold with unique singular point $p_0 \defeq [0, 0, 0, 0, 1]$. Let
\begin{align*}
  D &\defeq \set{[z_0, \dotsc, z_4] \in V \colon z_0^8 + z_1^8 + z_2^8 + z_3^8 + z_4^2 = 0} \, \text{,} \\
  D^\prime &\defeq \set{[z_0, \dotsc, z_4] \in V \colon z_0^8 - z_1^8 + 2 z_2^8 - 2 z_3^8 + \I z_4^2 = 0} \, \text{,} \quad \text{and} \\
  \Sigma &\defeq D \cap D^\prime \, \text{.}
\end{align*}
Note that both $D$ and $D^\prime$ are smooth and that $\Sigma$ is a complete intersection. Let $\tilde{V}$ be the blow-up of~$V$ along~$\Sigma$. Then $D$ lifts to a submanifold~$\tilde{D}$ of~$\tilde{V}$ which is isomorphic to~$D$.

Define two antiholomorphic involutions
\begin{align*}
  &\rho_1 \colon V \to V \, \text{,} \quad [z_0, \dotsc, z_4] \mapsto [\conjugate{z_1}, - \conjugate{z_0}, \conjugate{z_3}, - \conjugate{z_2}, \conjugate{z_4}] \quad \text{and} \\
  &\rho_2 \colon V \to V \, \text{,} \quad [z_0, \dotsc, z_4] \mapsto [\conjugate{z_1}, \conjugate{z_0}, \conjugate{z_3}, \conjugate{z_2}, \conjugate{z_4}] \, \text{.}
\end{align*}
Note that $\rho_1 \rho_2 = \rho_2 \rho_1$. Furthermore, both $\rho_1$ and $\rho_2$ preserve $p_0$, $D$, and $D^\prime$, and hence also $\Sigma$. So they lift to antiholomorphic involutions $\tilde{\rho}_1$ and $\tilde{\rho}_2$ of~$\tilde{V}$ such that $\tilde{\rho}_1 \tilde{\rho}_2 = \tilde{\rho}_2 \tilde{\rho}_1$. Note further that $\rho_1$ fixes only~$p_0$. Let $M_1 \defeq (\tilde{V} \setminus (\tilde{D} \cup \set{p_0})) / \inner{\tilde{\rho}_1}$.

Now let $W$ be the blow-up of $\C^4 / \Z_4$ at~$0$, where $\Z_4$ acts on~$\C^4$ by multiplication by~$\I$. Furthermore, let
\begin{align*}
  \rho_2^\prime \colon \C^4 / \Z_4 &\to \C^4 / \Z_4 \, \text{,} \quad [z_1, \dotsc, z_4] \mapsto [\conjugate{z_1}, \conjugate{z_2}, \conjugate{z_3}, \conjugate{z_4}] \, \text{,} \\
  \rho_3^\prime \colon \C^4 / \Z_4 &\to \C^4 / \Z_4 \, \text{,} \quad [z_1, \dotsc, z_4] \mapsto [\conjugate{z_2}, - \conjugate{z_1}, \conjugate{z_4}, - \conjugate{z_3}] \, \text{,}
\end{align*}
and let $\tilde{\rho}_2^\prime$ and $\tilde{\rho}_3^\prime$ be the lifts to~$W$. Then $\tilde{\rho}_3^\prime$ is an involution of~$W$ that acts freely. Let $M_2 \defeq W / \inner{\tilde{\rho}_3^\prime}$.

Let $M \defeq M_1 \cup_f M_2$, where we glue $M_1$ and $M_2$ together using the following diffeomorphism~$f$ of $S^7 / Q_8$, where $Q_8$ is the quaternion group. Let $f^\prime \colon \C^2 \to \C^2$,
\begin{equation*}
  f^\prime(x_1 + \I x_2, x_3 + \I x_4) \defeq (- x_1 + \I x_3, x_2 + \I x_4) \, \text{,}
\end{equation*}
and define $f \defeq (f^\prime \oplus f^\prime) \vert_{S^7 / Q_8}$. Then $M$ admits an asymptotically cylindrical Riemannian metric with holonomy $\Spin(7)$ \cite[Theorem~5.9]{Kov13}. Note that the cross-section at infinity of~$M$ is diffeomorphic to $N \defeq (S^1 \times D) / \inner{\rho}$, where $\rho(\E^{\I \theta}, z) = (\E^{- \I \theta}, \rho_1(z))$.

We may assume that the Kähler metric on~$\tilde{V}$ is $\tilde{\rho}_1$\dash- and $\tilde{\rho}_2$\dash-invariant. Then the holomorphic volume form~$\Omega$ on~$\tilde{V}$ satisfies (\wolog) $(\tilde{\rho}_1)^\ast(\Omega) = \overline{\Omega}$ and $(\tilde{\rho}_2)^\ast(\Omega) = \E^{\I \theta} \overline{\Omega}$ for some $\theta \in \R$. Now $\tilde{\rho}_1 \tilde{\rho}_2 = \tilde{\rho}_2 \tilde{\rho}_1$ implies $\E^{2 \I \theta} = 1$. So either $(\tilde{\rho}_2)^\ast(\Omega) = \overline{\Omega}$ or $(\tilde{\rho}_2)^\ast(\Omega) = - \overline{\Omega}$. In the first case, the fixed-point set of~$\tilde{\rho}_2$ is calibrated with respect to $\Re \Omega$, and in the second case, the fixed-point set of~$\tilde{\rho}_2$ is calibrated with respect to $\Im \Omega$. In particular, $\tilde{\rho}_1$ acts orientation-preserving in the first case and orientation-reversing in the second case.

So consider the fixed-point set of~$\rho_2$ in~$V$,
\begin{equation*}
  V_{\rho_2} \defeq \set{[u, \overline{u}, v, \overline{v}, \tau] \colon u, v \in \C, \tau \in \R, (u, v, \tau) \ne (0, 0, 0)} \, \text{.}
\end{equation*}
The action of~$\rho_1$ on~$V_{\rho_2}$ is given by
\begin{equation*}
  \rho_1([u, \overline{u}, v, \overline{v}, \tau]) = [\I u, \overline{\I u}, \I v, \overline{\I v}, \tau] \, \text{.}
\end{equation*}
In particular, $\rho_1$ acts orientation-preserving on $V_{\rho_2}$ (i.e.,~the quotient $V_{\rho_2} / \inner{\rho_1}$ is orientable). This shows that $(\tilde{\rho}_2)^\ast(\Omega) = \overline{\Omega}$ as seen above.

By resolving the singularity in a $\tilde{\rho}_2$\dash-equivariant way, $\tilde{\rho}_2$ yields an involution of~$M$ that preserves the $\Spin(7)$-structure. Hence its fixed-point set is a Cayley submanifold of~$M$ \cite[Proposition~10.8.6]{Joy00}. Denote it by~$X$. We will now deduce the topological type of~$X$ and calculate the index of the deformation map.

Let $\tilde{X}_1$ be the fixed-point set of~$\tilde{\rho}_2$ in~$\tilde{V} \setminus \tilde{D}$, let $\tilde{X}_2$ be the fixed-point set of~$\tilde{\rho}_1 \tilde{\rho}_2$ in~$\tilde{V} \setminus \tilde{D}$, let $\tilde{X}_3$ be the fixed-point set of $\tilde{\rho}_2^\prime \tilde{\rho}_3^\prime$ in~$W$, and let $X_1 \defeq \tilde{X}_1 / \inner{\tilde{\rho}_1}$, $X_2 \defeq \tilde{X}_2 / \inner{\tilde{\rho}_1}$, and $X_3 \defeq \tilde{X}_3 / \inner{\tilde{\rho}_3}$.

\begin{lemma}
  We have
  \begin{equation*}
    X \cong X_1 \cup_{L(4, 1)} X_2 \cup_{L(4; 1)} X_3 \, \text{,}
  \end{equation*}
  where $L(4; 1)$ is a lens space, defined by considering $S^3 \subseteq \C^2$ and taking the quotient with respect to the cyclic group~$\Z_4$ whose action is induced by multiplication by~$\I$.
\end{lemma}

\begin{proof}
  We have
  \begin{align*}
    f^\prime(u, \bar{u}) &= ((- 1 + \I) \Re u, (1 - \I) \Im u) \, \text{,} \\
    f^\prime(\I u, \bar{u}) &= (\I \bar{u}, \bar{u}) \, \text{,} \quad \text{and} \\
    f^\prime(u, 0) &= (- \Re u, \Im u)
  \end{align*}
  for all $u \in \C$. So
  \begin{equation*}
    \set{[u, \bar{u}, v, \bar{v}, 1] \colon u, v \in \C} \subseteq V_{\rho_2}
  \end{equation*}
  is identified with
  \begin{equation*}
    \set{[(1 - \I) t_1, (1 - \I) t_2, (1 - \I) t_3, (1 - \I) t_4] \colon t_1, \dotsc, t_4 \in \R} \subseteq \C^4 / \Z_4 \, \text{,}
  \end{equation*}
  and
  \begin{equation*}
    \set{[u, \bar{u}, v, \bar{v}, -1] \colon u, v \in \C} = \set{[\I u, \bar{u}, \I v, \bar{v}, 1] \colon u, v \in \C} \subseteq V_{\rho_2}
  \end{equation*}
  is identified with
  \begin{equation*}
    \set{[\I u, u, \I v, v] \colon u, v \in \C} \subseteq \C^4 / \Z_4 \, \text{,}
  \end{equation*}
  and
  \begin{equation*}
    \set{[u, 0, v, 0, 1] \colon u, v \in \C} \subseteq V_{\rho_1 \rho_2}
  \end{equation*}
  is identified with
  \begin{equation*}
    \set{[t_1, t_2, t_3, t_4] \colon t_1, \dotsc, t_4 \in \R} \subseteq \C^4 / \Z_4 \, \text{.}
  \end{equation*}
  Now $[\I t_1, \I t_2, \I t_3, \I t_4]$ and $[(1 - \I) t_1, (1 - \I) t_2, (1 - \I) t_3, (1 - \I) t_4]$ lie on the same complex curve. So under the blow-up of $\C^4 / \Z_4$ at~$0$, they are glued together. Furthermore, the component $\set{[u, 0, v, 0, 1] \colon u, v \in \C}$ is blown up at~$p_0$ (this is a complex surface in $\C^4 / \Z_4$ with unique singular point~$p_0$).
\end{proof}

Let $\tilde{Y}$ be the fixed-point set of~$\rho_2$ in~$D$, let $\tilde{Z}$ be the fixed-point set of~$\rho_1 \rho_2$ in~$D$, and let $Y \defeq \tilde{Y} / \inner{\rho_1}$ and $Z \defeq \tilde{Z} / \inner{\rho_1}$. Then the cross-section at infinity of~$X$ consists of $S^1 \times Z$ and two copies of~$Y$.

Note that $\tilde{Y}$ is a special Lagrangian submanifold of~$D$. Furthermore, the pull-back of $\normal[N]{Y}$ under the map $\tilde{Y} \to Y$ is $\normal[S^1 \times D]{\tilde{Y}}$. So $s = \frac{\partial}{\partial \theta}$ is a parallel normal vector field of $\tilde{Y}$ in $S^1 \times D$. Furthermore, $s \circ \rho = - s$. So we can apply parts of \propositionref{prop:parallel-section-lift} for the index calculation.

\begin{lemma} \label{lemma:Y-diffeomorphism-type}
  We have
  \begin{equation*}
    Y \cong 13 \, (S^1 \times S^2) \, \text{.}
  \end{equation*}
\end{lemma}

\begin{proof}
  Let
  \begin{equation*}
    \tilde{X}_4 \defeq \set{[u, \bar{u}, v, \bar{v}, \tau] \in V_{\rho_2} \colon \Re u^8 + \Re v^8 + \tfrac{1}{2} \tau^2 \le 0}
  \end{equation*}
  and $X_4 \defeq \tilde{X}_4 / \inner{\rho_1}$. Note that
  \begin{equation*}
    \tilde{Y} = \set{[u, \bar{u}, v, \bar{v}, \tau] \in V_{\rho_2} \colon \Re u^8 + \Re v^8 + \tfrac{1}{2} \tau^2 = 0} \, \text{.}
  \end{equation*}
  The map
  \begin{equation*}
    \tilde{X}_4 \times [0, 1] \to \tilde{X}_4 \, \text{,} \quad ([u, \bar{u}, v, \bar{v}, \tau], t) \mapsto [u, \bar{u}, v, \bar{v}, \tau t])
  \end{equation*}
  defines a deformation retraction of~$\tilde{X}_4$ onto
  \begin{equation*}
    \tilde{Y}^\prime \defeq \set{[u, \bar{u}, v, \bar{v}, 0] \in V_{\rho_2} \colon \Re u^8 + \Re v^8 \le 0} \, \text{.}
  \end{equation*}
  Note that $\tilde{Y}$ is the closed double of~$\tilde{Y}^\prime$. Furthermore, there exists a deformation retraction $f \colon \C \times [0, 1] \to \C$ onto
  \begin{equation*}
    \set{z \in \C \colon z^8 = - \abs{z}^8}
  \end{equation*}
  such that
  \begin{compactenum}[(i)]
    \item $f(\E^{\frac{1}{4} \pi \I} z, t) = \E^{\frac{1}{4} \pi \I} f(z, t)$,
    \item $f(r z, t) = r f(z, t)$, and
    \item $f(z, t) = 0$ implies $z^8 = \abs{z}^8$
  \end{compactenum}
  for all $z \in \C, t \in [0, 1], r \in [0, \infty)$. Then the map
  \begin{equation*}
    \tilde{Y}^\prime \times [0, 1] \to \tilde{Y}^\prime \, \text{,} \quad ([u, \bar{u}, v, \bar{v}, 0], t) \mapsto [f(u, t), \overline{f(u, t)}, f(v, t), \overline{f(v, t)}, 0])
  \end{equation*}
  defines a deformation retraction of~$\tilde{Y}^\prime$ onto
  \begin{equation*}
    \tilde{K} \defeq \set{[u, \bar{u}, v, \bar{v}, 0] \in V_{\rho_2} \colon \Re u^8 = - \abs{u}^8, \Re v^8 = - \abs{v}^8} \, \text{.}
  \end{equation*}
  Now let $Y^\prime \defeq \tilde{Y}^\prime / \inner{\rho_1}$ and $K \defeq \tilde{K} / \inner{\rho_1}$. The above deformation retractions are $\rho_1$\dash-invariant. So $K$ is a deformation retract of~$X_4$, and $Y$ is the closed double of~$Y^\prime$. Note that $K$ and $\tilde{K}$ are $1$\dash-dimensional CW\dash-complexes. Furthermore, $Y$ and $\tilde{Y}$ are connected since $K$ and $\tilde{K}$ are connected.
  
  Now the set
  \begin{equation*}
    \set{(u, v) \in \C^2 \colon \abs{u}^8 + \abs{v}^8 = 1, \Re u^8 = - \abs{u}^8, \Re v^8 = - \abs{v}^8}
  \end{equation*}
  is a $1$\dash-dimensional CW-complex with $16$~vertices ($0$\dash-cells) and $64$~edges ($1$\dash-cells). So $K$ has $4$~vertices and $16$~edges. Hence $\chi(K) = 4 - 16 = - 12$, and so $b^1(K) = b^0(K) - \chi(K) = 13$. In particular, $\pi_1(Y^\prime) \cong \pi_1(K)$ is a free group of rank~$13$.
  
  Given any closed loop in~$Y^\prime$, there exists a smooth loop homotopic to it. Furthermore, we may assume that it meets $K$ transversely, which means that it does not intersect~$K$. Then we may use the earlier deformation retraction backwards to push the loop onto~$\partial Y^\prime$. This shows that the map $\pi_1(\partial Y^\prime) \to \pi_1(Y^\prime)$ induced by the inclusion is surjective. Hence
  \begin{equation*}
    \pi_1(Y) \cong \pi_1(Y^\prime) \mathbin\ast_{\pi_1(\partial Y^\prime)} \pi_1(Y^\prime) \cong \pi_1(Y^\prime)
  \end{equation*}
  by van Kampen's Theorem. So $\pi_1(Y)$ is a free group of rank~$13$, and hence $Y \cong 13 \, (S^1 \times S^2)$ by \cite[Exercise~5.3]{Hem76}.
\end{proof}

We should also note the following consequences of the above proof. We have $\sigma(X_4) = \sigma(\tilde{X}_4) = 0$ since $b^2(X_4) = b^2(K) = 0$ and $b^2(\tilde{X}_4) = b^2(\tilde{K}) = 0$. Furthermore, $\tilde{Y} \cong 25 \, (S^1 \times S^2)$ since $\tilde{K}$ has $8$~vertices and $32$~edges. Also $b^1(\partial Y^\prime) = 2 \, b^1(Y^\prime) = 26$. So $\partial Y^\prime$ is a closed orientable surface of genus~$13$.

Consider also the fixed-point set of~$\rho_1 \rho_2$ in~$V$,
\begin{align*}
  V_{\rho_1 \rho_2} &\defeq \set{[u, 0, v, 0, w] \colon u, v, w \in \C, (u, v, w) \neq (0, 0, 0)} \quad \text{and} \\
  V_{\rho_1 \rho_2}^\prime &\defeq \set{[0, u, 0, v, w] \colon u, v, w \in \C, (u, v, w) \neq (0, 0, 0)} \, \text{.}
\end{align*}
Note that $\rho_1$ interchanges these two components.

\begin{lemma}
  The closed orientable surface~$Z$ has genus~$3$. Furthermore,
  \begin{equation*}
    \dim_\C H^0(Z, \normal[D]{Z}) = 4 \, \text{.}
  \end{equation*}
\end{lemma}

\begin{proof}
  Note that
  \begin{equation*}
    Z \cong \set{[u, 0, v, 0, w] \in V_{\rho_1 \rho_2} \colon u^8 + v^8 + w^2 = 0} \, \text{.}
  \end{equation*}
  The map
  \begin{equation*}
    Z \to \CP^1 \, \text{,} \quad [u, 0, v, 0, w] \mapsto [u, v]
  \end{equation*}
  is a branched double cover with~$8$ branched points. So $\chi(Z) = - 4$ by the Riemann--Hurwitz Theorem. Hence $Z$ is a complex curve of genus~$3$.
  
  In particular, $c_1(\normal[D]{Z}) = - c_1(Z) = 4$ since $D$ is a Calabi--Yau manifold. In fact, if we denote the above map by $f \colon Z \to \CP^1$, then one can check that $\normal[D]{Z} \cong f^\ast \mathcal{O}(1) \oplus f^\ast \mathcal{O}(1)$. Note that $\dim_\C H^0(\CP^1, \mathcal{O}(1)) = 2$. If we denote $\rho \colon Z \to Z$, $[u, 0, v, 0, w] \mapsto [u, 0, v, 0, - w]$, then these sections are those sections that are invariant under~$\rho$. But if a section changed its sign under~$\rho$, then it would be~$0$ at all $8$ branched points. This would imply that the intersection number with the zero section would be at least~$8$. But this intersection number is equal to $c_1(f^\ast \mathcal{O}(1)) = \frac{1}{2} c_1(\normal[D]{Z}) = 2$. Hence $\dim_\C H^0(Z, \normal[D]{Z}) = 4$.
\end{proof}

Let $\tilde{S}$ be the fixed-point set of~$\rho_2$ in~$\Sigma$, and let $S \defeq \tilde{S} / \inner{\rho_1}$.

\begin{lemma}
  The closed orientable surface~$S$ has genus~$13$.
\end{lemma}

\begin{proof}
  Note that
  \begin{equation*}
    S \cong \set{[u, \bar{u}, v, \bar{v}, \tau] \in V_{\rho_2} / \inner{\rho_1} \colon \Re u^8 + \Re v^8 + \tfrac{1}{2} \tau^2 = \Im u^8 + 2 \Im v^8 + \tfrac{1}{2} \tau^2 = 0} \, \text{.}
  \end{equation*}
  We have
  \begin{align*}
    \Re u^8 + \Re v^8 = \Im u^8 + 2 \Im v^8 \quad &\Leftrightarrow \quad \Re((1 + \I) u^8) + \Re((1 + 2 \I) v^8) = 0 \\
    &\Leftrightarrow \quad \Re(\lambda u)^8 + \Re(\mu v)^8 = 0 \, \text{,}
  \end{align*}
  where $\lambda, \mu \in \C$ are such that $\lambda^8 = 1 + \I$ and $\mu^8 = 1 + 2 \I$. So $S$ is diffeomorphic to the closed double of
  \begin{equation*}
    S^\prime \defeq \set{[u, \bar{u}, v, \bar{v}, 0] \in V_{\rho_2} / \inner{\rho_1} \colon \Re u^8 + \Re v^8 \le 0, \Re(\lambda u)^8 + \Re(\mu v)^8 = 0} \, \text{.}
  \end{equation*}
  Since the map
  \begin{equation*}
    [u, \bar{u}, v, \bar{v}, 0] \mapsto [\E^{\frac{1}{8} \pi \I} u, \E^{- \frac{1}{8} \pi \I} \bar{u}, \E^{\frac{1}{8} \pi \I} v, \E^{- \frac{1}{8} \pi \I} \bar{v}, 0]
  \end{equation*}
  defines a diffeomorphism of~$S^\prime$ to
  \begin{equation*}
    \set{[u, \bar{u}, v, \bar{v}, 0] \in V_{\rho_2} / \inner{\rho_1} \colon \Re u^8 + \Re v^8 \ge 0, \Re(\lambda u)^8 + \Re(\mu v)^8 = 0} \, \text{,}
  \end{equation*}
  we see that
  \begin{equation*}
    S \cong \set{[u, \bar{u}, v, \bar{v}, 0] \in V_{\rho_2} / \inner{\rho_1} \colon \Re(\lambda u)^8 + \Re(\mu v)^8 = 0} \, \text{.}
  \end{equation*}
  This surface is diffeomorphic to~$\partial Y^\prime$. Hence $S$ is a closed orientable surface of genus~$13$.
\end{proof}

Define
\begin{equation*}
  \tilde{X}_4 \defeq \set{[u, \bar{u}, v, \bar{v}, \tau] \in V_{\rho_2} \colon \Re u^8 + \Re v^8 + \tfrac{1}{2} \tau^2 \le 0}
\end{equation*}
and $X_4 \defeq \tilde{X}_4 / \inner{\rho_1}$ as in the proof of \lemmaref{lemma:Y-diffeomorphism-type}. Furthermore, let
\begin{equation*}
  \tilde{X}_5 \defeq \set{[u, \bar{u}, v, \bar{v}, \tau] \in V_{\rho_2} \colon \Re u^8 + \Re v^8 + \tfrac{1}{2} \tau^2 \ge 0}
\end{equation*}
and $X_5 \defeq \tilde{X}_5 / \inner{\rho_1}$.

Note that $V_{\rho_2} \setminus D$ consists of two connected components, namely the interiors of $\tilde{X}_4$ and $\tilde{X}_5$. In the blow-up $\tilde{V} \setminus \tilde{D}$, they are glued together along an interval~$I$ times the intersection of $V_{\rho_2}$ with~$\Sigma$, that is, along $I \times \tilde{S}$. So
\begin{equation*}
  X_1 \cong X_4 \cup_{I \times S} X_5 \, \text{.}
\end{equation*}
Note that by construction, we also have a diffeomorphism
\begin{equation*}
  X_4 \cup_Y X_5 \cong \R \times L(4; 1) \, \text{.}
\end{equation*}
Hence
\begin{equation*}
  \chi(X_1) = \chi(X_4) + \chi(X_5) - \chi(S) = \chi(L(4; 1)) + \chi(Y) - \chi(S) = 24 \, \text{.}
\end{equation*}
Also note that $X_1$ has an orientation-reversing diffeomorphism (coming from $V \to V$, $[z_0, \dotsc, z_4] \mapsto [z_0, \dotsc, z_3, - z_4]$), and hence
\begin{equation*}
  \sigma(X_1) = 0 \, \text{.}
\end{equation*}

The intersection of $V_{\rho_1 \rho_2}$ with~$\Sigma$, that is, the set
\begin{equation*}
  \set{[u, 0, v, 0, w] \in V \colon u^8 + v^8 + w^2 = 0, u^8 + 2 v^8 + \I w^2 = 0} \, \text{,}
\end{equation*}
consists of $16$~points. So in the blow-up~$\tilde{V}$, the submanifold $\CP^2_{1, 1, 4} \setminus \set{p_0}$ will be blown up at~$16$ points. Hence
\begin{equation*}
  X_2 \cup_{S^1 \times Z} (D^2 \times Z) \cong (\CP^2_{1, 1, 4} \setminus \set{p_0}) \mathbin\# 16 \, \overline{\CP^2} \, \text{.}
\end{equation*}

\begin{proposition} \label{prop:top-type-first-example}
  The topological type of~$X$ is
  \begin{equation*}
    (\CP^2 \mathbin\# 17 \, \overline{\CP^2}) \setminus (X_{13} \amalg X_{13} \amalg (D^2 \times \Sigma_3)) \, \text{,}
  \end{equation*}
  where $\Sigma_3$ is a closed orientable surface of genus~$3$ and $X_{13}$ is the null-cobordism of $13 \, (S^1 \times S^2)$ coming from $S^1 \times S^2 \cong \partial(S^1 \times D^3)$. Note that the cross-section at infinity is
  \begin{equation*}
    13 \, (S^1 \times S^2) \amalg 13 \, (S^1 \times S^2) \amalg (S^1 \times \Sigma_3) \, \text{.}
  \end{equation*}
\end{proposition}

\begin{proof}
  The relation $X_4 \cup_Y X_5 \cong \R \times L(4; 1)$ implies
  \begin{equation*}
    \pi_1(X_4) \mathbin\ast_{\pi_1(Y)} \pi_1(X_5) \cong \pi_1(L(4; 1)) \cong \Z_4
  \end{equation*}
  by van Kampen's Theorem. Since the inclusion $Y \hookrightarrow X_4$ induces an isomorphism $\pi_1(Y) \cong \pi_1(X_4)$, we must have $\pi_1(X_5) \cong \Z_4$. In fact, the inclusion $L(4; 1) \hookrightarrow X_5$ induces this isomorphism. Now $X_1 \cong X_4 \cup_{I \times \Sigma_{13}} X_5$ implies
  \begin{equation*}
    \pi_1(X_1) \cong \pi_1(X_4) \mathbin\ast_{\pi_1(\Sigma_{13})} \pi_1(X_5)
  \end{equation*}
  by van Kampen's Theorem. Here the map $\pi_1(\Sigma_{13}) \to \pi_1(X_4)$ induced by the inclusion $I \times \Sigma_{13} \hookrightarrow X_4$ is surjective. So the map $\pi_1(L(4; 1)) \to \pi_1(X_1)$ induced by the inclusion $L(4; 1) \hookrightarrow X_5 \hookrightarrow X_1$ is surjective.
  
  Note that $\CP^2_{1, 1, 4} \setminus \set{p_0}$ is simply-connected \cite[Corollary~8]{DD85}. Hence
  \begin{equation*}
    X_2 \cup_{S^1 \times Z} (D^2 \times Z) \cong (\CP^2_{1, 1, 4} \setminus \set{p_0}) \mathbin\# 16 \, \overline{\CP^2}
  \end{equation*}
  is simply-connected. Furthermore, $X_3$ is homotopy equivalent to~$\CP^1$ (the exceptional divisor), and hence also $X_3$ is simply-connected.
  
  So if we define
  \begin{equation*}
    \widebar{X} \defeq X \cup_Y X_{13} \cup_Y X_{13} \cup_{S^1 \times Z} (D^2 \times Z) \, \text{,}
  \end{equation*}
  then $\pi_1(\widebar{X}) = 0$ by van Kampen's Theorem since $X \cong X_1 \cup_{L(4; 1)} X_2 \cup_{L(4; 1)} X_3$ and the map $\pi_1(Y) \to \pi_1(X_{13})$ induced by the inclusion is surjective.
  
  We calculate
  \begin{align*}
    \chi(\widebar{X}) &= \chi(X_1) + 2 \chi(X_{13}) + \chi(\CP^2_{1, 1, 4} \setminus \set{p_0}) + 16 \cdot (\chi(\overline{\CP^2}) - 2) + \chi(X_3) \\
    &= 24 + 2 \cdot (1 - 13) + 2 + 16 \cdot (3 - 2) + 2 = 20
  \end{align*}
  since the Euler characteristic of a closed orientable $3$\dash-manifold is~$0$, and
  \begin{align*}
    \sigma(\widebar{X}) &= \sigma(X_1) + 2 \sigma(X_{13}) + \sigma(\CP^2_{1, 1, 4} \setminus \set{p_0}) + 16 \sigma(\overline{\CP^2}) + \sigma(X_3) \\
    &= 0 + 2 \cdot 0 + 1 + 16 \cdot (-1) + (-1) = -16
  \end{align*}
  by Novikov additivity since $b^2(X_{13}) = 0$. Also note that the intersection form on~$\widebar{X}$ is odd since the intersection form on $\overline{\CP^2}$ is odd. Hence $\widebar{X}$ is homeomorphic to $\CP^2 \mathbin\# 17 \, \overline{\CP^2}$ by \cite[Theorem~10.1\,(2)]{FQ90}.
\end{proof}

Note that
\begin{align*}
  \chi(X) &= \chi(\CP^2) + 17 \cdot (\chi(\overline{\CP^2}) - 2) - 2 \chi(X_{13}) - \chi(D^2 \times \Sigma_3) \\
  &= 3 + 17 \cdot (3 - 2) - 2 \cdot (1 - 13) - (-4) = 48
\end{align*}
and
\begin{equation*}
  \sigma(X) = \sigma(\CP^2) + 17 \cdot \sigma(\overline{\CP^2}) - 2 \sigma(X_{13}) - \sigma(D^2 \times \Sigma_3) = -16 \, \text{.}
\end{equation*}

\begin{lemma} \label{lemma:euler-class-normal-bundle-first-example}
  We have
  \begin{equation*}
    \int_X e(\normal[M]{X}) = 24 \, \text{.}
  \end{equation*}
\end{lemma}

\begin{proof}
  When $X_1$, $X_2$, and $X_3$ are glued together along $L(4; 1)$, the normal bundles are also glued together. Now fix a non-vanishing section of the normal bundle $\normal[S^7 / Q_8]{L(4; 1)}$ (by abuse of notation, we will also call this~$s$). Then
  \begin{equation*}
    \int_X e(\normal[M]{X}) = \frac{1}{2} e(\pi^\ast \normal[\tilde{V} / \inner{\rho_1}]{X_1}, s) + e(\normal[\tilde{V} / \inner{\rho_1}]{X_2}, s) + e(\normal[W / \inner{\rho_3}]{X_3}, s) \, \text{,}
  \end{equation*}
  where $\pi \colon \tilde{V} \to \tilde{V} / \inner{\rho_1}$ is the natural projection.
  
  Let $\hat{V}$ be the blow-up of~$\tilde{V}$ at the singular point~$p_0$, and let $\hat{X}_1$ be the fixed-point set of~$\hat{\rho}_2$ (the extension of $\rho_2$ to~$\hat{V}$) in $\hat{V} \setminus \tilde{D}$. Then
  \begin{equation*}
    e(\pi^\ast \normal[\tilde{V} / \inner{\rho_1}]{X_1}, s) = e(\normal[\hat{V}]{\hat{X}_1}, s) = \chi(\hat{X}_1) = \chi(\tilde{X}_1) = 2 \chi(X_1) = 48
  \end{equation*}
  since $\hat{X}_1$ is a special Lagrangian submanifold of~$\hat{V}$ and
  \begin{equation*}
    \hat{X}_1 \cong \tilde{X}_1 \cup_{L(4; 1) \amalg L(4; 1)} (I \times L(4; 1)) \, \text{.}
  \end{equation*}
  
  Furthermore, let $F_4$ be the blow-up of $\CP^2_{1, 1, 4}$ at the singularity~$p_0$. Then $X_2 \cup_{L(4; 1)} X_3$ is $F_4$ blown up at $16$~points. Denote it by~$\hat{F}_4$. We can view $\hat{F}_4$ as a submanifold of~$\hat{V}$. Then
  \begin{equation*}
    e(\normal[\tilde{V} / \inner{\rho_1}]{X_2}, s) + e(\normal[W / \inner{\rho_3}]{X_3}, s) = [\hat{F}_4] \cdot [\hat{F}_4] \, \text{,}
  \end{equation*}
  the self-intersection number of~$\hat{F}_4$ inside~$\hat{V}$. Note that $[F_4] \cdot [F_4] = 0$ inside the blow-up of~$V$ at the singularity~$p_0$ since the map
  \begin{equation*}
    \CP^2_{1, 1, 4} \times [0, 1] \to \CP^2_{1, 1, 1, 1, 4} \, \text{,} \quad ([u, v, w], t) \mapsto [u, t u, v, t v, w]
  \end{equation*}
  extends to the blow-ups at the singularity~$p_0$. In fact, this map also extends to $\hat{F}_4 \times [0, 1] \to \hat{V}$. So $[\hat{F}_4] \cdot [\hat{F}_4] = 0$. Hence
  \begin{equation*}
    \int_X e(\normal[M]{X}) = 24 \, \text{.} \qedhere
  \end{equation*}
\end{proof}

\begin{lemma}
  We have
  \begin{equation*}
    \ind_\lambda D = -22 \, \text{,}
  \end{equation*}
  where $\lambda > 0$ is such that $[- \lambda, 0)$ contains no eigenvalue of~$\tilde{D}$.
\end{lemma}

\begin{proof}
  We have
  \begin{align*}
    \ind_\lambda D &= \frac{1}{2} \chi(X) + \frac{1}{2} \sigma(X) - \int_X e(\normal[M]{X}) + \frac{1}{2} (2 \sigma(X_4) - \sigma(\tilde{X}_4)) \\
    &\phantom{{}={}} {}+ (b^0(Y) + b^1(Y) - b^0(\tilde{Y}) - b^1(\tilde{Y})) - \frac{1}{2} \dim_\C H^0(Z, \normal[D]{Z}) \\
    &= \frac{1}{2} \cdot 48 + \frac{1}{2} \cdot (-16) - 24 + \frac{1}{2} \cdot (2 \cdot 0 - 0) + (1 + 13 - 1 - 25) - \frac{1}{2} \cdot 4 \\
    &= -22
  \end{align*}
  by \theoremref{thm:main-index-formula} and parts of \propositionref{prop:parallel-section-lift} and \propositionref{prop:complex}. Note that the change of sign of~$\sigma(X)$ in the above formula compared to~\eqref{eq:index-eta} comes from our convention~\eqref{eq:def-Phi-0} of the $\Spin(7)$-structure (complex surfaces of Calabi--Yau $4$\dash-folds are Cayley with respect to the opposite orientation).
\end{proof}

\begin{remark}
  Note that
  \begin{equation*}
    \dim \ker (\tilde{D} \vert_{13 \, (S^1 \times S^2)}) = 12 \quad \text{and} \quad \dim \ker (\tilde{D} \vert_{S^1 \times \Sigma_3}) = 4 \, \text{.}
  \end{equation*}
  So all connected components of the cross-section at infinity are obstructed as associative submanifolds.
\end{remark}

\subsection{Cayley Submanifolds in a Spin(7)-Manifold Constructed from a Hypersurface in \texorpdfstring{$\mathbf{CP^5_{1, 1, 1, 1, 4, 4}}$}{CP\textasciicircum 5\_\{1,1,1,1,4,4\}}}
\label{subsec:example-2}

We will now construct two asymptotically cylindrical Cayley submanifolds inside the asymptotically cylindrical $\Spin(7)$-manifold constructed in \cite[Section~6.3]{Kov13} and compute the index. Let
\begin{align*}
  V &\defeq \set{[z_0, \dotsc, z_5] \in \CP^5_{1, 1, 1, 1, 4, 4} \colon z_0^8 + z_1^8 + z_2^8 + z_3^8 + z_4^2 + z_5^2} \, \text{,} \\
  C &\defeq \set{[z_0, \dotsc, z_5] \in \CP^5_{1, 1, 1, 1, 4, 4} \colon z_4 + z_5 = 0} \, \text{,} \\
  C^\prime &\defeq \set{[z_0, \dotsc, z_5] \in \CP^5_{1, 1, 1, 1, 4, 4} \colon z_4 - z_5 = 0} \, \text{,} \\
  D &\defeq V \cap C \, \text{,} \quad \text{and} \\
  \Sigma &\defeq V \cap C \cap C^\prime \, \text{.}
\end{align*}
Note that $V$ is an orbifold with two singular points $p_\pm \defeq [0, 0, 0, 0, \I, \pm 1]$ and that $D$ is smooth. Let $\tilde{V}$ be the blow-up of~$V$ along~$\Sigma$. Then $D$ lifts to a submanifold~$\tilde{D}$ of~$\tilde{V}$ which is isomorphic to~$D$.

Define two antiholomorphic involutions
\begin{align*}
  &\rho_1 \colon [z_0, \dotsc, z_5] \mapsto [\conjugate{z_1}, - \conjugate{z_0}, \conjugate{z_3}, - \conjugate{z_2}, \conjugate{z_5}, \conjugate{z_4}] \quad \text{and} \\
  &\rho_2 \colon [z_0, \dotsc, z_5] \mapsto [\conjugate{z_1}, \conjugate{z_0}, \conjugate{z_3}, \conjugate{z_2}, \conjugate{z_5}, \conjugate{z_4}] \, \text{.}
\end{align*}
Note that $\rho_1 \rho_2 = \rho_2 \rho_1$. Furthermore, both $\rho_1$ and $\rho_2$ preserve $p_\pm$, $V$, $C$, $C^\prime$, and hence also $D$ and $\Sigma$. So they lift to antiholomorphic involutions $\tilde{\rho}_1$ and $\tilde{\rho}_2$ of~$\tilde{V}$ such that $\tilde{\rho}_1 \tilde{\rho}_2 = \tilde{\rho}_2 \tilde{\rho}_1$. Note further that $\rho_1$ fixes only~$p_\pm$. Let $M_1 \defeq (\tilde{V} \setminus (\tilde{D} \cup \set{p_\pm})) / \inner{\tilde{\rho}_1}$.

Define also $M_2$ and $f \colon S^7 / Q_8 \to S^7 / Q_8$ as in the \hyperref[subsec:example-1]{last section}. Then both $M \defeq M_1 \cup_f M_2 \cup_f M_2$ and $M^\prime \defeq M_1 \cup_{\id} M_2 \cup_f M_2$ admit asymptotically cylindrical Riemannian metrics with holonomy $\Spin(7)$ \cite[Theorem~5.9]{Kov13}. Note that the cross-section at infinity of~$M$ and~$M^\prime$ is diffeomorphic to $N \defeq (S^1 \times D) / \inner{\rho}$, where $\rho(\E^{\I \theta}, z) = (\E^{- \I \theta}, \rho_1(z))$.

Consider the fixed-point set of~$\rho_2$ in~$V$,
\begin{equation*}
  V_{\rho_2} \defeq \set{[u, \overline{u}, v, \overline{v}, w, \overline{w}] \colon u, v, w \in \C, \Re u^8 + \Re v^8 + \Re w^2 = 0} \, \text{.}
\end{equation*}
The action of~$\rho_1$ on~$V_{\rho_2}$ is given by
\begin{equation*}
  \rho_1([u, \overline{u}, v, \overline{v}, w, \overline{w}]) = [\I u, \overline{\I u}, \I v, \overline{\I v}, w, \overline{w}] \, \text{.}
\end{equation*}
In particular, $\rho_1$ acts orientation-preserving on $V_{\rho_2}$ (i.e.,~the quotient $V_{\rho_2} / \inner{\rho_1}$ is orientable). This shows that $(\tilde{\rho}_2)^\ast(\Omega) = \overline{\Omega}$ as seen in the \hyperref[subsec:example-1]{last section}.

Note that
\begin{equation*}
  \set{[u, \overline{u}, v, \overline{v}, w, \overline{w}] \in V_{\rho_2} \setminus \set{p_\pm} \colon \Re w \ge 0} \cong \tilde{X}_5 \, \text{,}
\end{equation*}
where $\tilde{X}_5$ was defined in the \hyperref[subsec:example-1]{last section}. So
\begin{equation*}
  V_{\rho_2} \setminus \set{p_\pm} \cong \tilde{X}_5 \cup_Y \tilde{X}_5 \, \text{.}
\end{equation*}
Note that in the blow-up~$\tilde{V}$, we get
\begin{equation*}
  X_5 \cup_{I \times \Sigma_{13}} X_5
\end{equation*}
as one part of our Cayley submanifold.

Consider also the fixed-point set of~$\rho_1 \rho_2$ in~$V$,
\begin{align*}
  V_{\rho_1 \rho_2} &\defeq \set{[u, 0, v, 0, w, z] \colon u, v, w, z \in \C, u^8 + v^8 + w^2 + z^2 = 0} \quad \text{and} \\
  V_{\rho_1 \rho_2}^\prime &\defeq \set{[0, u, 0, v, w, z] \colon u, v, w, z \in \C, u^8 + v^8 + w^2 + z^2 = 0} \, \text{.}
\end{align*}
Note that $\rho_1$ interchanges these two components.

We have $h^{2, 0}(V_{\rho_1 \rho_2}) = 0$ and $h^{1, 1}(V_{\rho_1 \rho_2}) = 8$ by \cite[Theorem~7.2]{IF00}. So $\chi(V_{\rho_1 \rho_2}) = 10$ and $\sigma(V_{\rho_1 \rho_2}) = -6$. Note also that
\begin{equation*}
  V_{\rho_1 \rho_2} \setminus (\set{p_\pm} \cup \set{[0, 0, v, 0, w, z] \in V_{\rho_1 \rho_2} \colon v^8 + w^2 + z^2 = 0})
\end{equation*}
is simply-connected by \cite[Lemma~9]{DD85}. Hence $V_{\rho_1 \rho_2} \setminus \set{p_\pm}$ is simply-connected since
\begin{equation*}
  \set{[0, 0, v, 0, w, z] \in V_{\rho_1 \rho_2} \colon v^8 + w^2 + z^2 = 0} \to \CP^1 \, \text{,} \quad [0, 0, v, 0, w, z] \mapsto [w, z]
\end{equation*}
is an $8$\dash-fold branched cover with $2$~branched points, and hence
\begin{equation*}
  \set{[0, 0, v, 0, w, z] \in V_{\rho_1 \rho_2} \colon v^8 + w^2 + z^2 = 0}
\end{equation*}
is diffeomorphic to~$S^2$ (in particular, simply-connected).

\begin{proposition} \label{prop:top-type-second-example}
  Suppose that we resolve both singularities by the ``non-trivial'' resolution (the manifold~$M$). Then the topological type of~$X$ is
  \begin{equation*}
    (13 \, \CP^2 \mathbin\# 29 \, \overline{\CP^2}) \setminus (X_{13} \amalg X_{13} \amalg (D^2 \times \Sigma_3)) \, \text{,}
  \end{equation*}
  where $\Sigma_3$ is a closed orientable surface of genus~$3$ and $X_{13}$ is the null-cobordism of $13 \, (S^1 \times S^2)$ coming from $S^1 \times S^2 \cong \partial(S^1 \times D^3)$. Note that the cross-section at infinity is
  \begin{equation*}
    13 \, (S^1 \times S^2) \amalg 13 \, (S^1 \times S^2) \amalg (S^1 \times \Sigma_3) \, \text{.}
  \end{equation*}
\end{proposition}

\begin{proof}
  Similarly to the proof of \propositionref{prop:top-type-first-example}, one can check that
  \begin{equation*}
    \widebar{X} \defeq X \cup_Y X_{13} \cup_Y X_{13} \cup_{S^1 \times Z} (D^2 \times Z)
  \end{equation*}
  is simply-connected. We calculate
  \begin{align*}
    \chi(\overline{X}) &= 2 \chi(X_5) - \chi(\Sigma_{13}) + 2 \chi(X_{13}) + \chi(V_{\rho_1 \rho_2} \setminus \set{p_\pm}) \\
    &\phantom{{}={}} {}+ 8 \cdot (\chi(\overline{\CP^2}) - 2) + 2 \chi(X_3) \\
    &= 2 \cdot 12 - (-24) + 2 \cdot (-12) + 8 + 16 \cdot (3 - 2) + 2 \cdot 2 = 44
  \end{align*}
  and
  \begin{align*}
    \sigma(\overline{X}) &= \sigma(X_5 \cup_Y X_5) + \sigma(V_{\rho_1 \rho_2} \setminus \set{p_\pm}) + 8 \sigma(\overline{\CP^2}) + 2 \sigma(X_3) \\
    &= 0 + (-6) + 8 \cdot 1 + 2 \cdot (-1) = -16 \, \text{.}
  \end{align*}
  Furthermore, the intersection form on~$\widebar{X}$ is odd. Hence $\widebar{X}$ is homeomorphic to $13 \, \CP^2 \mathbin\# 29 \, \overline{\CP^2}$ by \cite[Theorem~10.1\,(2)]{FQ90}.
\end{proof}

Note that
\begin{align*}
  \chi(X) &= 2 + 13 \cdot (\chi(\CP^2) - 2) + 29 \cdot (\chi(\overline{\CP^2}) - 2) - 2 \chi(X_{13}) - \chi(D^2 \times \Sigma_3) \\
  &= 2 + 13 \cdot (3 - 2) + 29 \cdot (3 - 2) - 2 \cdot (1 - 13) - (-4) = 72
\end{align*}
and
\begin{equation*}
  \sigma(X) = 13 \sigma(\CP^2) + 29 \cdot \sigma(\overline{\CP^2}) - 2 \sigma(X_{13}) - \sigma(D^2 \times \Sigma_3) = -16 \, \text{.}
\end{equation*}

\begin{proposition}
  Suppose that we resolve both singularities by different resolutions (the manifold~$M^\prime$). Then the topological type of~$X$ is
  \begin{equation*}
    ((S^1 \times S^3) \mathbin\# 14 \, \CP^2 \mathbin\# 30 \, \overline{\CP^2}) \setminus (X_{13} \amalg X_{13} \amalg (D^2 \times \Sigma_3)) \, \text{,}
  \end{equation*}
  where $\Sigma_3$ is a closed orientable surface of genus~$3$ and $X_{13}$ is the null-cobordism of $13 \, (S^1 \times S^2)$ coming from $S^1 \times S^2 \cong \partial(S^1 \times D^3)$. Note that the cross-section at infinity is
  \begin{equation*}
    13 \, (S^1 \times S^2) \amalg 13 \, (S^1 \times S^2) \amalg (S^1 \times \Sigma_3) \, \text{.}
  \end{equation*}
\end{proposition}

\begin{proof}
  Let
  \begin{equation*}
    \widebar{X} \defeq X \cup_Y X_{13} \cup_Y X_{13} \cup_{S^1 \times Z} (D^2 \times Z) \, \text{.}
  \end{equation*}
  Note that each~$X_5$ in the splitting $X_5 \cup_{I \times \Sigma_{13}} X_5$ approaches both singularities. Recall that
  \begin{equation*}
    X_5 \cong \set{[u, \bar{u}, v, \bar{v}, \tau] \in \text{``$V_{\rho_2} / \inner{\rho_1}$''} \colon \Re u^8 + \Re v^8 + \tfrac{1}{2} \tau^2 \ge 0} \, \text{,}
  \end{equation*}
  and consider
  \begin{equation*}
    X_5^\prime \defeq \set{[u, \bar{u}, v, \bar{v}, \tau] \in X_5 \colon \tau \ge 0} \, \text{.}
  \end{equation*}
  Note also that
  \begin{equation*}
    \Sigma_{13} \cong \set{[u, \bar{u}, v, \bar{v}, \tau] \in X_5 \colon \tau = 0} \, \text{.}
  \end{equation*}
  Note that the remaining part of the boundary of~$X_5^\prime$ is given by
  \begin{equation*}
    \set{[u, \bar{u}, v, \bar{v}, 0] \in X_5 \colon \Re u^8 + \Re v^8 \le 0} \cong Y^\prime \, \text{.}
  \end{equation*}
  So
  \begin{equation*}
    X_5^\prime \cup_{I \times \Sigma_{13}} X_5^\prime \cong X_5 \, \text{.}
  \end{equation*}
  In other words, we may assume that one~$X_5$ in the splitting $X_5 \cup_{I \times \Sigma_{13}} X_5$ approaches~$p_+$ and the other approaches~$p_-$. So that component looks like
  \begin{equation*}
    (X_5 \cup_{L(4; 1) \amalg L(4; 1)} (I \times L(4; 1))) \cup_{I \times \Sigma_{13}} X_5
  \end{equation*}
  Now recall that $X_4 \cup_Y X_5 \cong \R \times L(4; 1)$. Hence
  \begin{equation*}
    X_4 \cup_Y (X_5 \cup_{L(4; 1) \amalg L(4; 1)} (I \times L(4; 1))) \cong S^1 \times L(4; 1) \, \text{,}
  \end{equation*}
  which has fundamental group $\Z \times \Z_4$. Note further that the map $\pi_1(\Sigma_{13}) \to \pi_1(X_5)$ induced by the inclusion $I \times \Sigma_{13} \hookrightarrow X_5$ is surjective by a similar argument as in the end of the proof of \lemmaref{lemma:Y-diffeomorphism-type} (noting also that $\pi_1(\Sigma_{13}) \to \pi_1(Y^\prime)$ induced by the inclusion is surjective). Hence $\pi_1(\widebar{X}) \cong \Z$ by van Kampen's Theorem.
  
  We have $\chi(\widebar{X}) = 44$ and $\sigma(\widebar{X}) = -16$ as seen in the proof of \propositionref{prop:top-type-second-example}. Furthermore, the intersection form is odd. Hence $\widebar{X}$ is homeomorphic to
  \begin{equation*}
    (S^1 \times S^3) \mathbin\# 14 \, \CP^2 \mathbin\# 30 \, \overline{\CP^2}
  \end{equation*}
  by \cite[Corollary~3\,(2)]{HT97}.
\end{proof}

Note that the map
\begin{align*}
  V_{\rho_1 \rho_2} \times [0, 1] &\to V \, \text{,} \\
  ([u, 0, v, 0, w, z], t) &\mapsto [(1 - t^8)^{1/8} u, t u, (1 - t^8)^{1/8} v, t v, w, z]
\end{align*}
extends to the blow-ups, which shows that the self-intersection number is~$0$ (similar to the proof of \lemmaref{lemma:euler-class-normal-bundle-first-example}). Hence
\begin{equation*}
  \int_X e(\normal[M]{X}) = \chi(X_5 \cup_{I \times \Sigma_{13}} X_5) = 2 \chi(X_5) - \chi(\Sigma_{13}) = 2 \cdot 12 - (-24) = 48 \, \text{.}
\end{equation*}

\begin{lemma}
  We have
  \begin{equation*}
    \ind_\lambda D = -28 \, \text{,}
  \end{equation*}
  where $\lambda > 0$ is such that $[- \lambda, 0)$ contains no eigenvalue of~$\tilde{D}$.
\end{lemma}

\begin{proof}
  We have
  \begin{align*}
    \ind_\lambda D &= \frac{1}{2} \chi(X) + \frac{1}{2} \sigma(X) - \int_X e(\normal[M]{X}) + \frac{1}{2} (2 \sigma(X_4) - \sigma(\tilde{X}_4)) \\
    &\phantom{{}={}} {}+ (b^0(Y) + b^1(Y) - b^0(\tilde{Y}) - b^1(\tilde{Y})) - \frac{1}{2} \dim_\C H^0(Z, \normal[D]{Z}) \\
    &= \frac{1}{2} \cdot 72 + \frac{1}{2} \cdot (-16) - 48 + \frac{1}{2} \cdot (2 \cdot 0 - 0) + (1 + 13 - 1 - 25) - \frac{1}{2} \cdot 4 \\
    &= -28
  \end{align*}
  by \theoremref{thm:main-index-formula} and parts of \propositionref{prop:parallel-section-lift} and \propositionref{prop:complex}.  (complex surfaces of Calabi--Yau $4$\dash-folds are Cayley with respect to the opposite orientation).
\end{proof}

\section{Relation to Other Calibrations}
\label{sec:relation-other-calibrations}

Harvey and Lawson noted in \cite[Remark~2.12 in Chapter~IV]{HL82} that the geometry of Cayley submanifolds includes the geometries of other calibrations. In particular, if the holonomy reduces to a proper subgroup of~$\Spin(7)$, then Cayley submanifolds can be constructed out of submanifolds that are calibrated with respect to another calibration. An application of the volume-minimising property of calibrated submanifolds shows that any deformation of such a closed Cayley submanifold as a Cayley submanifold must again be of that form. Here we show that this is also true for asymptotically cylindrical Cayley submanifolds. We further simplify the index formulae in these cases.

In particular, in the case of the special Lagrangian and the coassociative calibration, the moduli space of asymptotically cylindrical Cayley deformations is a smooth manifold (we have the same linearisation up to isomorphism but better control on the non-linear terms due to Hodge theory).

\subsection{Special Lagrangian Calibration}
\label{subsec:special-lagrangian}

Let $M$ be an asymptotically cylindrical Calabi--Yau $4$\dash-fold with Kähler form~$\omega$ and holomorphic volume form~$\Omega$, which we assume to be \emph{normalised}, that is, $\omega^4 = \frac{3}{2} \Omega \wedge \bar{\Omega}$. Then
\begin{equation*}
  \Phi \defeq - \frac{1}{2} \, \omega \wedge \omega + \Re \Omega
\end{equation*}
defines a $\Spin(7)$\dash-structure on~$M$ \cite[Proposition~1.32 in Chapter~IV]{HL82}. This $\Spin(7)$\dash-structure is asymptotically cylindrical and torsion-free since $\omega$ and $\Omega$ are asymptotically cylindrical and closed.

An orientable $4$\dash-dimensional submanifold~$X$ of~$M$ is called \emph{special Lagrangian} if $(\Re \Omega) \vert_X = \vol_X$ for some orientation of~$X$. This is equivalent to $\omega \vert_X = 0$, $(\Im \Omega) \vert_X = 0$ \cite[Corollary~1.11 in Chapter~III]{HL82}. So every special Lagrangian submanifold is Cayley, but not every Cayley submanifold is special Lagrangian (for example, complex surfaces are also Cayley).

\begin{proposition}
  Let $M$ be an asymptotically cylindrical Calabi--Yau $4$\dash-fold, let $X$ be an asymptotically cylindrical special Lagrangian submanifold of~$M$, and let $Y$ be an asymptotically cylindrical local deformation of~$X$ with the same asymptotic limit as~$X$.
  
  If $Y$ is a Cayley submanifold of~$M$, then $Y$ is a special Lagrangian submanifold of~$M$. So the moduli space of all local deformations of~$X$ as an asymptotically cylindrical Cayley submanifold of~$M$ with the same asymptotic limit as~$X$ can be identified with the moduli space of all local deformations of~$X$ as an asymptotically cylindrical special Lagrangian submanifold of~$M$ with the same asymptotic limit as~$X$.
\end{proposition}

\begin{proof}
  We have
  \begin{equation*}
    \lim_{T \to \infty} \bigl(\vol(\set{y \in Y \colon t \le T}) - \vol(\set{x \in X \colon t \le T})\bigr) = 0
  \end{equation*}
  by \propositionref{prop:minimal} with $\varphi = \Phi$. So \propositionref{prop:minimal} with $\varphi = \Re \Omega$ implies that $Y$ is special Lagrangian.
\end{proof}

Now suppose that $M$ is an asymptotically cylindrical Calabi--Yau $4$\dash-fold with cross-section~$N$, and that $X$ is an asymptotically cylindrical special Lagrangian submanifold of~$M$ with cross-section~$Y$. The complex structure yields an isomorphism $T X \cong \normal[M]{X}$. Furthermore, $\frac{\partial}{\partial t}$~is mapped to a parallel section $s \in \sections(\normal[N]{Y})$ under this isomorphism. In particular, $e(\normal[M]{X}, s) = e(T X, \frac{\partial}{\partial t}) = \chi(X)$. So the index in this case is
\begin{equation}
  \ind_\lambda D = - \frac{1}{2} \chi(X) - \frac{1}{2} \sigma(X) - \frac{b^0(Y) + b^1(Y)}{2}
\end{equation}
by \propositionref{prop:parallel-section}, where $\lambda > 0$ is such that $[- \lambda, 0)$ contains no eigenvalue of~$\Bev$.

Note that we also have an isomorphism $\altforms^0 X \oplus \altforms^2_+ X \cong E$, which comes from the parallel section $\omega \in \sections(\altforms^2_7 M)$ and the map $T X \times T X \to E$, $(u, v) \mapsto u \times J v$. Under these isomorphisms, the Dirac operator~$D$ is identified with
\begin{equation}
  \forms^1(X) \to \forms^0(X) \oplus \forms^2_+(X) \, \text{,} \quad \alpha \mapsto (\updelta \alpha, \tfrac{1}{2} (\D \alpha + \hodgestar \D \alpha)) \, \text{.}
\end{equation}
The dimension of the kernel of this map is equal to the dimension of the image of $H^1_{\text{cs}}(X)$ in $H^1(X)$, which is equal to $b^3(X) + b^0(X) - b^4(X) - b^0(Y)$ (note that $b^4(X) = 0$ if $X$ has no closed connected component).

Salur and Todd \cite[Theorem~1.1]{ST10} proved that if $M$ is an asymptotically cylindrical Calabi--Yau $3$\dash-fold and $X$ is an asymptotically cylindrical special Lagrangian submanifold of~$M$, then the moduli space of all local deformations of~$X$ as an asymptotically cylindrical special Lagrangian submanifold of~$M$ with the same asymptotic limit as~$X$ is a smooth manifold whose dimension is given by the dimension of the image of $H^1_{\text{cs}}(X)$ in $H^1(X)$ (which is equal to $b^2(X) + b^0(X) - b^3(X) - b^0(Y)$ in dimension~$3$). The same is true in higher dimensions.

\subsection{Coassociative Calibration}
\label{subsec:coassociative}

Let $\tilde{M}$ be an asymptotically cylindrical $7$\dash-manifold with an asymptotically cylindrical $G_2$\dash-structure $\tilde{\varphi}$, let $\tilde{\psi}$ be the Hodge-dual of~$\tilde{\varphi}$, let $M \defeq S^1 \times \tilde{M}$, and let $\theta$ denote the coordinate on the $S^1$\dash-factor. Then
\begin{equation*}
  \Phi \defeq \D \theta \wedge \tilde{\varphi} + \tilde{\psi}
\end{equation*}
defines a $\Spin(7)$-structure on~$M$ \cite[Proposition~1.30 in Chapter~IV]{HL82}. The $\Spin(7)$-structure $\Phi$ is asymptotically cylindrical since $\tilde{\varphi}$ is asymptotically cylindrical. Furthermore, the $\Spin(7)$-structure $\Phi$ is torsion-free if the $G_2$\dash-structure~$\tilde{\varphi}$ is torsion-free.

An orientable $4$\dash-dimensional submanifold~$\tilde{X}$ of~$\tilde{M}$ is called \emph{coassociative} if $\tilde{\psi} \vert_{\tilde{X}} = \vol_{\tilde{X}}$ for some orientation of~$\tilde{X}$. This is equivalent to $\tilde{\varphi} \vert_{\tilde{X}} = 0$ \cite[Corollary~1.20 in Chapter~IV]{HL82}. So $\tilde{X}$ is a coassociative submanifold of~$\tilde{M}$ if and only if $X \defeq \set{1} \times \tilde{X}$ is a Cayley submanifold of~$M$.

\begin{proposition}
  Let $\tilde{M}$ be an asymptotically cylindrical $7$\dash-manifold with an asymptotically cylindrical torsion-free $G_2$\dash-structure, and let $\tilde{X}$ be an asymptotically cylindrical coassociative submanifold of~$\tilde{M}$. Define $M \defeq S^1 \times \tilde{M}$ and $X \defeq \set{1} \times \tilde{X}$. Furthermore, let $Y$ be an asymptotically cylindrical local deformation of~$X$ with the same asymptotic limit as~$X$.
  
  If $Y$ is a Cayley submanifold of~$M$, then $Y = \set{1} \times \tilde{Y}$ for some asymptotically cylindrical coassociative submanifold~$\tilde{Y}$ of~$\tilde{M}$. So the moduli space of all local deformations of~$X$ as an asymptotically cylindrical Cayley submanifold of~$M$ with the same asymptotic limit as~$X$ can be identified with the moduli space of all local deformations of~$\tilde{X}$ as an asymptotically cylindrical coassociative submanifold of~$\tilde{M}$ with the same asymptotic limit as~$\tilde{X}$.
\end{proposition}

\begin{proof}
  Let $\pi \colon M = S^1 \times \tilde{M} \to \tilde{M}$ be the projection. If $Z$ is an asymptotically cylindrical local deformation of~$X$, then the volume of $Z_T \defeq \set{z \in Z \colon t \le T}$ is greater than or equal to the volume of $\pi(Z_T)$, and equality holds if and only if $Z_T$ is a submanifold of $\set{p} \times \tilde{M}$ for some $p \in S^1$. We have
  \begin{align*}
    \lim_{T \to \infty} \bigl(\vol(\set{y \in Y \colon t \le T}) - \vol(\set{x \in X \colon t \le T})\bigr) &= 0 \quad \text{and} \\
     \lim_{T \to \infty} \bigl(\vol(\set{y \in \set{1} \times \pi(Y) \colon t \le T}) - \vol(\set{x \in X \colon t \le T})\bigr) &\ge 0
  \end{align*}
  by \propositionref{prop:minimal} with $\varphi = \Phi$, and hence
  \begin{equation*}
    \lim_{T \to \infty} \bigl(\vol(\set{y \in Y \colon t \le T}) - \vol(\set{y \in \set{1} \times \pi(Y) \colon t \le T})\bigr) \le 0 \, \text{.}
  \end{equation*}
  So $Y = \set{p} \times \tilde{Y}$ for some $p \in S^1$ and some submanifold~$\tilde{Y}$ of~$\tilde{M}$ since $\vol(Y_T) - \vol(\pi(Y_T))$ is increasing as $T \to \infty$. We have $p = 1$ since $Y$ and $X$ have the same asymptotic limit. Furthermore, $\tilde{Y}$ is an asymptotically cylindrical coassociative submanifold since $Y$ is an asymptotically cylindrical Cayley submanifold.
\end{proof}

Now suppose that $\tilde{M}$ is an asymptotically cylindrical $7$\dash-manifold with cross-section~$\tilde{N}$ and with an asymptotically cylindrical torsion-free $G_2$\dash-structure, and that $\tilde{X}$ is an asymptotically cylindrical coassociative submanifold of~$\tilde{M}$ with cross-section~$\tilde{Y}$. Define $M \defeq S^1 \times \tilde{M}$, $N \defeq S^1 \times \tilde{N}$, $X \defeq \set{1} \times \tilde{X}$, and $Y \defeq \set{1} \times \tilde{Y}$. Let $\theta$ denote the coordinate on the $S^1$\dash-factor. Then $\frac{\partial}{\partial \theta}$ is a non-trivial parallel section of~$\normal[M]{X}$. So $e(\normal[M]{X}, \frac{\partial}{\partial \theta} \vert_Y) = 0$, and hence the index in this case is
\begin{equation}
  \ind_\lambda D = \frac{1}{2} \chi(X) - \frac{1}{2} \sigma(X) - \frac{b^0(Y) + b^1(Y)}{2}
\end{equation}
by \propositionref{prop:parallel-section}, where $\lambda > 0$ is such that $[- \lambda, 0)$ contains no eigenvalue of~$\Bev$.

Note that we have isomorphisms $\normal[M]{X} \cong \altforms^2_- X \oplus \altforms^4 X$ and $E \cong T X$, which come from taking the interior product with~$\varphi$ and the cross-product with~$\frac{\partial}{\partial \theta}$. Under these isomorphisms, the Dirac operator~$D$ is identified with
\begin{equation}
  \forms^2_-(X) \oplus \forms^4(X) \to \forms^3(X) \, \text{,} \quad (\alpha, \beta) \mapsto \D \alpha + \updelta \beta \, \text{.}
\end{equation}
The dimension of the kernel of this map is equal to the dimension of the negative subspace of the image of $H^2_{\text{cs}}(X)$ in $H^2(X)$.

Joyce and Salur \cite[Theorem~1.1]{JS05} proved that the moduli space of all local deformations of~$\tilde{X}$ as an asymptotically cylindrical coassociative submanifold of~$\tilde{M}$ with the same asymptotic limit as~$\tilde{X}$ is a smooth manifold whose dimension is given by the dimension of the negative subspace of the image of $H^2_{\text{cs}}(X)$ in $H^2(X)$. Note that they get the dimension of the positive subspace since they use a different convention for the $\Spin(7)$-structure (not our convention~\eqref{eq:def-Phi-0}).

Examples of asymptotically cylindrical coassociative submanifolds inside asymptotically cylindrical Riemannian $7$\dash-manifolds with holonomy~$G_2$ were constructed by Kovalev and Nordström in \cite[Section~5.3]{KN10}.

\subsection{Complex Surfaces}
\label{subsubsec:complex-surfaces}

Recall from \sectionref{subsec:special-lagrangian} that if $M$ is an asymptotically cylindrical Calabi--Yau $4$\dash-fold with Kähler form~$\omega$ and holomorphic volume form~$\Omega$, then $\Phi \defeq - \frac{1}{2} \omega \wedge \omega + \Re \Omega$ defines an asymptotically cylindrical torsion-free $\Spin(7)$\dash-structure on~$M$.

Furthermore, every complex surface in~$M$ is Cayley, but not every Cayley submanifold is a complex surface (for example, special Lagrangian submanifolds are also Cayley).

\begin{proposition}
  Let $M$ be an asymptotically cylindrical Calabi--Yau $4$\dash-fold, let $X$ be an asymptotically cylindrical complex surface in~$M$, and let $Y$ be an asymptotically cylindrical local deformation of~$X$ with the same asymptotic limit as~$X$.
  
  If $Y$ is a Cayley submanifold of~$M$, then $Y$ is a complex surface in~$M$. So the moduli space of all local deformations of~$X$ as an asymptotically cylindrical Cayley submanifold of~$M$ with the same asymptotic limit as~$X$ can be identified with the moduli space of all local deformations of~$X$ as an asymptotically cylindrical complex surface in~$M$ with the same asymptotic limit as~$X$.
\end{proposition}

\begin{proof}
  We have
  \begin{equation*}
    \lim_{T \to \infty} \bigl(\vol(\set{y \in Y \colon t \le T}) - \vol(\set{x \in X \colon t \le T})\bigr) = 0
  \end{equation*}
  by \propositionref{prop:minimal} with $\varphi = \Phi$. So \propositionref{prop:minimal} with $\varphi = - \frac{1}{2} \omega \wedge \omega$ implies that $Y$ is a complex surface.
\end{proof}

Now suppose that $M$ is an asymptotically cylindrical Calabi--Yau $4$\dash-fold with cross-section~$N$, and that $X$ is an asymptotically cylindrical complex surface in~$M$ with cross-section~$Y$. Assume that $N \cong S^1 \times D$ for some Calabi--Yau $3$\dash-fold~$D$. Then also $Y \cong S^1 \times C$ for some complex curve~$C$ in~$D$. Write $\widebar{M}$ and $\widebar{X}$ for the compactifications as in \remarkref{rmk:complex-compactification}. So $M \cong \widebar{M} \setminus D$ and $X \cong \widebar{X} \setminus C$. Then $\chi(X) = \chi(\widebar{X}) - \chi(C)$ and $\sigma(X) = \sigma(\widebar{X})$ since $\normal[\widebar{X}]{C}$ is trivial. So the index in this case is
\begin{equation}
  \ind_\lambda D = \frac{1}{2} \chi(\widebar{X}) + \frac{1}{2} \sigma(\widebar{X}) - [\widebar{X}] \cdot [\widebar{X}] - \frac{1}{2} \chi(C) - \frac{1}{2} \dim_\C H^0(C, \normal[D]{C})
\end{equation}
by \propositionref{prop:complex} and \remarkref{rmk:complex-compactification}, where $\lambda > 0$ is such that $[- \lambda, 0)$ contains no eigenvalue of~$\tilde{D}$. Note that here $\sigma(\widebar{X})$ refers to the orientation that is induced by the complex structure.

\subsection{Associative Calibration}
\label{subsec:associative}

Recall from \sectionref{subsec:coassociative} that if $\tilde{M}$ is an asymptotically cylindrical $7$\dash-manifold with an asymptotically cylindrical $G_2$\dash-structure~$\tilde{\varphi}$, $\tilde{\psi}$~is the Hodge-dual of~$\tilde{\varphi}$, $M \defeq S^1 \times \tilde{M}$, and $\theta$ denotes the coordinate on the $S^1$\dash-factor, then $\Phi \defeq \D t \wedge \tilde{\varphi} + \tilde{\psi}$ defines an asymptotically cylindrical $\Spin(7)$-structure on~$M$. Furthermore, the $\Spin(7)$-structure~$\Phi$ is torsion-free if the $G_2$\dash-structure~$\tilde{\varphi}$ is torsion-free.

An orientable $3$\dash-dimensional submanifold~$\tilde{X}$ of~$\tilde{M}$ is called \emph{associative} if $\tilde{\varphi} \vert_{\tilde{X}} = \vol_{\tilde{X}}$ for some orientation of~$\tilde{X}$. So $\tilde{X}$ is an associative submanifold of~$\tilde{M}$ if and only if $X \defeq S^1 \times \tilde{X}$ is a Cayley submanifold of~$M$.

\begin{proposition}
  Let $\tilde{M}$ be an asymptotically cylindrical $7$\dash-manifold with an asymptotically cylindrical torsion-free $G_2$\dash-structure, and let $\tilde{X}$ be an asymptotically cylindrical associative submanifold of~$\tilde{M}$. Define $M \defeq S^1 \times \tilde{M}$ and $X \defeq S^1 \times \tilde{X}$. Furthermore, let $Y$ be an asymptotically cylindrical local deformation of~$X$ with the same asymptotic limit as~$X$.
  
  If $Y$ is a Cayley submanifold of~$M$, then $Y = S^1 \times \tilde{Y}$ for some asymptotically cylindrical associative submanifold~$\tilde{Y}$ of~$\tilde{M}$. So the moduli space of all local deformations of~$X$ as an asymptotically cylindrical Cayley submanifold of~$M$ with the same asymptotic limit as~$X$ can be identified with the moduli space of all local deformations of~$\tilde{X}$ as an asymptotically cylindrical associative submanifold of~$\tilde{M}$ with the same asymptotic limit as~$\tilde{X}$.
\end{proposition}

\begin{proof}
  For $p \in S^1$, let $Y_p \defeq Y \cap (\set{p} \times \tilde{M})$. Note that $Y_p$ is an asymptotically cylindrical submanifold of~$\tilde{M}$ for all $p \in S^1$ as $Y$ is a local deformation of~$X$. Then
  \begin{equation*}
    \vol(\set{y \in Y \colon t \le T}) \ge \int_{S^1} \vol(\set{y \in Y_{\E^{\I \theta}} \colon t \le T}) \, \D \theta \, \text{,}
  \end{equation*}
  and equality holds if and only if $Y = S^1 \times Y_1$. Furthermore,
  \begin{equation*}
    \lim_{T \to \infty} \bigl(\vol(\set{y \in Y_p \colon t \le T}) - \vol(\set{x \in \tilde{X} \colon t \le T})\bigr) \ge 0
  \end{equation*}
  by \propositionref{prop:minimal} with $\varphi = \tilde{\varphi}$. Here we have equality if and only if $Y_p$ is associative. But
  \begin{equation*}
    \lim_{T \to \infty} \bigl(\vol(\set{y \in Y \colon t \le T}) - \vol(\set{x \in X \colon t \le T})\bigr) = 0
  \end{equation*}
  by \propositionref{prop:minimal} with $\varphi = \Phi$. Hence we get equality in the above inequalities (noting that the difference in the first inequality is increasing as $T \to \infty$). So $Y = S^1 \times Y_1$, and $Y_1$ is an asymptotically cylindrical associative submanifold of~$\tilde{M}$.
\end{proof}

Now suppose that $\tilde{M}$ is an asymptotically cylindrical $7$\dash-manifold with cross-section~$\tilde{N}$ and with an asymptotically cylindrical torsion-free $G_2$\dash-structure, and that $\tilde{X}$ is an asymptotically cylindrical coassociative submanifold of~$\tilde{M}$ with cross-section~$\tilde{Y}$. Define $M \defeq S^1 \times \tilde{M}$, $N \defeq S^1 \times \tilde{N}$, $X \defeq S^1 \times \tilde{X}$, and $Y \defeq S^1 \times \tilde{Y}$. Let $\theta$ denote the coordinate on the $S^1$\dash-factor. Then the cross-product with the parallel section $\frac{\partial}{\partial \theta} \in \sections(T X)$ defines an isomorphism $\normal[M]{X} \cong E$. Under this isomorphism, the Dirac operator~$D$ is self-adjoint. Furthermore, $\tilde{N}$ is a Calabi--Yau $3$\dash-fold and $\tilde{Y}$ is a complex curve. Hence
\begin{equation}
  \ind_\lambda D = - \frac{1}{2} \dim_\C H^0(\tilde{Y}, \normal[\tilde{N}]{\tilde{Y}})
\end{equation}
by \cite[Theorem~7.4]{LM85} and \cite[Lemma~5.11]{CHNP12}, where $\lambda > 0$ is such that $[- \lambda, 0)$ contains no eigenvalue of~$\tilde{D}$.

\end{document}